\documentclass[11pt,a4paper, envcountsame,mathserif]{amsart}
\usepackage[usenames,dvipsnames]{color}

\usepackage[utf8]{inputenc}	
\usepackage{pdfpages,url}

 \usepackage[colorlinks,citecolor=blue,urlcolor=blue, linkcolor=blue, backref]{hyperref}

\usepackage[capitalise]{cleveref}		

 \usepackage{amsmath, amssymb, xspace, enumitem}
 \usepackage{graphicx}

 \usepackage[all]{xy}

\usepackage{tikz}
\usetikzlibrary{arrows,snakes,positioning,backgrounds,shadows}
\usepackage{stmaryrd}

\newtheorem{theorem}{Theorem}[section]

\newtheorem{thm}[theorem]{Theorem}

\newtheorem{fact}[theorem]{Fact}

\newtheorem{criterion}[theorem]{Criterion}

\newtheorem{prop}[theorem]{Proposition}

\newtheorem{claim}[theorem]{Claim}

\newtheorem{convention}[theorem]{Convention}

\newtheorem{lemma}[theorem]{Lemma}

\newtheorem{cor}[theorem]{Corollary}

\newtheorem{question}[theorem]{Question}

\theoremstyle{definition}
\newtheorem{definition}[theorem]{Definition}
\newtheorem{example}[theorem]{Example}
\newtheorem{remark}[theorem]{Remark}

\DeclareMathOperator \GL{GL}

\newcommand{\mrm}[1]{\mathrm{#1}}

\newcommand{\NN}{{\mathbb{N}}}
\newcommand{\RR}{{\mathbb{R}}}

\newcommand{\QQ}{{\mathbb{Q}}}
\newcommand{\ZZ}{{\mathbb{Z}}}
\newcommand{\CC}{{\mathbb{C}}}
\newcommand{\sub}{\subseteq}
\newcommand{\sN}[1]{_{#1\in \NN}}

\newcommand{\uhr}[1]{\! \upharpoonright_{#1}}
\newcommand{\ML}{Martin-L{\"o}f}
\newcommand{\SI}[1]{\Sigma^0_{#1}}

\newcommand{\bi}{\begin{itemize}}
\newcommand{\ei}{\end{itemize}}
\newcommand{\bc}{\begin{center}}
\newcommand{\ec}{\end{center}}

\newcommand{\ES}{\emptyset}

\newcommand{\tp}[1]{2^{#1}}

\newcommand{\ex}{\exists}
\newcommand{\fa}{\forall}
\newcommand{\lep}{\le^+}

\newcommand{\la}{\langle}
\newcommand{\ra}{\rangle}

\newcommand{\seqcantor}{2^{ \NN}}

\newcommand{\cantor}{\seqcantor}

\newcommand{\ROpcl}[2]{[#2]^\prec_{#1}}

\newcommand{\n}{\noindent}

\newcommand{\vsp}{\vspace{6pt}}

\newcommand{\sss}{\sigma}
\newcommand{\aaa}{\alpha}

\renewcommand{\S}{S_\infty}

\newcommand{\lland}{\, \land \, }

\newcommand\+[1]{\mathcal{#1}}

\newcommand{\wt}{\widetilde}
\newcommand{\ol}{\overline}
\newcommand{\ul}{\underline}
\newcommand{\ape}{\, \hat{\ } \, }

\newcommand{\lra}{\leftrightarrow}
\newcommand{\LR}{\Leftrightarrow}
\newcommand{\RA}{\Rightarrow}
\newcommand{\LA}{\Leftarrow}

\newcommand{\rapf}{\n $\RA:$\ }
\newcommand{\lapf}{\n $\LA:$\ }

\newcommand{\sssl}{\ensuremath{|\sigma|}}

\newcommand{\Sym}{\mrm{Sym}}

\DeclareMathOperator{\Aut}{Aut}
\DeclareMathOperator{\Inn}{Inn}
\DeclareMathOperator{\Out}{Out}

\DeclareMathOperator{\Ext}{\mathrm{Ext}}

\newcommand{\mor}[3]{{#2} \stackrel {#1} \to {#3}}

\numberwithin{equation}{section}
\renewcommand{\hat}{\widehat}
\begin{document}

\title{Logic Blog 2023-24}

 \author{Editor: Andr\'e Nies}

\maketitle

 
\setcounter{tocdepth}{1}
\tableofcontents

 {Note: the 2023 and 2024 Logic Blogs have been joined. 
 
The Logic   Blog is a shared platform for
\bi \item rapidly announcing  results and questions related to logic
\item putting up results and their proofs for further research
\item parking results for later use
\item getting feedback before submission to  a journal
\item fostering collaboration.   \ei

Each year's   blog is    posted on arXiv  2-3 months after the year has ended.

\vsp
\begin{tabbing} 

  \href{https://arxiv.org/pdf/2302.11853.pdf}{Logic Blog 2022} \ \ \ \   \= (Link: \texttt{http://arxiv.org/pdf/2302.11853.pdf})  \\

 \href{https://arxiv.org/pdf/2202.13643.pdf}{Logic Blog 2021} \ \ \ \   \= (Link: \texttt{http://arxiv.org/abs/2202.13643})  \\

 \href{https://arxiv.org/pdf/2101.09508.pdf}{Logic Blog 2020} \ \ \ \   \= (Link: \texttt{http://arxiv.org/abs/2101.09508})  \\
 
 \href{http://arxiv.org/pdf/2003.03361.pdf}{Logic Blog 2019} \ \ \ \   \= (Link: \texttt{http://arxiv.org/abs/2003.03361})  \\

 \href{http://arxiv.org/pdf/1902.08725.pdf}{Logic Blog 2018} \ \ \ \   \= (Link: \texttt{http://arxiv.org/abs/1902.08725})  \\
 
 \href{http://arxiv.org/pdf/1804.05331.pdf}{Logic Blog 2017} \ \ \ \   \= (Link: \texttt{http://arxiv.org/abs/1804.05331})  \\
 
 \href{http://arxiv.org/pdf/1703.01573.pdf}{Logic Blog 2016} \ \ \ \   \= (Link: \texttt{http://arxiv.org/abs/1703.01573})  \\
 
  \href{http://arxiv.org/pdf/1602.04432.pdf}{Logic Blog 2015} \ \ \ \   \= (Link: \texttt{http://arxiv.org/abs/1602.04432})  \\
  
  \href{http://arxiv.org/pdf/1504.08163.pdf}{Logic Blog 2014} \ \ \ \   \= (Link: \texttt{http://arxiv.org/abs/1504.08163})  \\

   \href{http://arxiv.org/pdf/1403.5719.pdf}{Logic Blog 2013} \ \ \ \   \= (Link: \texttt{http://arxiv.org/abs/1403.5719})  \\

    \href{http://arxiv.org/pdf/1302.3686.pdf}{Logic Blog 2012}  \> (Link: \texttt{http://arxiv.org/abs/1302.3686})   \\

 \href{http://arxiv.org/pdf/1403.5721.pdf}{Logic Blog 2011}   \> (Link: \texttt{http://arxiv.org/abs/1403.5721})   \\

 \href{http://dx.doi.org/2292/9821}{Logic Blog 2010}   \> (Link: \texttt{http://dx.doi.org/2292/9821})  
     \end{tabbing}

\vsp

\n {\bf How does the Logic Blog work?}

\vsp

\n {\bf Writing and editing.}  The source files are in a shared dropbox.
 Ask Andr\'e (\email{andre@cs.auckland.ac.nz})  in order    to gain access.

\vsp

\n {\bf Citing.}  Postings can be cited.  An example of a citation is:

\vsp

\n  H.\ Towsner, \emph{Computability of Ergodic Convergence}. In  Andr\'e Nies (editor),  Logic Blog, 2012, Part 1, Section 1, available at
\url{http://arxiv.org/abs/1302.3686}.}

%

%

 
The logic blog,  once it is on  arXiv,  produces citations e.g.\ on Google Scholar.
%

%
 
%

\part{Group theory and its connections to logic}

 \section{Algorithmic questions around the Tits alternative} 
  The 2023 \href{https://sites.google.com/view/nzmri2023?pli=1}{NZRMI summer school}  took place in Nelson, New Zealand. It was organized by Jeroen Schillewaert and Matthew Conder, both based in Auckland.   NZMRI means  ``\href{https://www.nzmri.org/}{New Zealand Mathematical Research Institute}".  There were  six  three-hour tutorials, by Emmanuel Breuillard (Oxford), Martin Bridson (Oxford), Benson Farb (U. Chicago),  Nicolas Monod (EPFL), Anne Thomas (U.\ Sydney), and Amie Wilkinson (U. Chicago).    Here, Andre Nies will discuss  some aspects of   Breuillard's lectures.

They  started out from the Tits alternative. 
 \begin{thm}[Tits, 1972] Let $K$ be a field, and let $\Gamma \le \GL_d(K)$ be finitely generated. Then either $\Gamma$ contains a non-abelian free subgroup, or $\Gamma$ has a solvable subgroup of finite index (i.e.,  $\Gamma$ is  virtually solvable). \end{thm}

 This  result implies that every f.g.\ torsion linear group is finite, because this is the case for   solvable groups as proved easily by induction          on the solvability rank.

 Breuillard in 2008 obtained a \href{https://arxiv.org/abs/0804.1395}{strong  Tits alternative}: given $d$ there is  $N(d)\in \NN$ such that for each field $K$ and each finite subset $S$ of $\GL_d(K)$ closed under inversion, either $\la S\ra$ is virtually solvable, or $S^{N(d)}$ contains elements $a,b$ generating a free non-abelian subgroup.

 Breuillard   mentioned the following open problem, which essentially  asks whether one  can  decide which side  of the Tits alternative holds.
 
 \begin{question} Given a     field $K$ and $d \ge 3$, is there an algorithm to decide whether non-commuting matrices $A,B \in \GL_d(K)$ generate a free group? \end{question} 
 The question appears to be open even for $K= \QQ$.  This question can be interpreted literally if $K$ is a computable field. It  is a bit harder to  make sense of it  for $K= \mathbb C$, the main case of interest.  In line with computable analysis, the answer will have to be computed  from approximations to $A,B$, such as matrices with entries that are in $\QQ(i)$. 
 
 \subsection{Torsion in $\GL_d(K)$}
Here is  a potentially  simpler question: 
Is there an algorithm to decide whether $A \in \GL_d(K)$ is a torsion element?

For algebraic number fields, the paper \href{https://dl.acm.org/doi/10.1145/164081.164104}{
Deciding finiteness of matrix groups} by Babai, Beals, Rockmore  in ISSAC 93  provides  an algorithm that decides the question in  polynomial time, and contains another   algorithm that does it in  Monte Carlo polynomial time with a more practical exponent. 
%
%
%
%
%
%
%

Let $\chi_A$ denote the characteristic polynomial $\chi_A(t)=\det (A - tI)$. Let $\Phi_k\in \ZZ[X]$ be the $k$-th cyclotomic polynomial. Thus, $\Phi_k$  is the minimal polynomial over $\QQ$ of a primitive $k$-th root of unity (such as $\exp(2\pi i /k)$)  and has degree $\varphi(k)$, where $\phi$ is the Euler totient function.  The following is known. 

\begin{fact} $A\in \GL_d(\CC)$ is torsion $\LR$ $\chi_A$ is a product of cyclotomic polynomials $\Phi_k$, where necessarily $\varphi(k) \le d$. \end{fact}

Since one can compute $\Phi_k$, and $\sqrt {k/2} \le \varphi(k)$ for  each $k\ge 1$, this gives   the desired algorithm when the coefficients of $A$ are taken from some computable field $K$ that is a subfield  of $\CC$ (i.e., $K$ is computable and has characteristics $0$): on input $A \in \GL_d(K)$, 

\bi 
\item[1.] compute $\chi_A$.

\item [2.] see if it can be written as a product of $\Phi_k$'s where  $k  \le 2d^2$. If so answer yes. \ei

Note that this also gives a crude upper bound on the torsion as a function of $d$, something like $2^d \cdot (d^{2d})$. Knowing that, as an  algorithm to determine whether $A$ is torsion, we can  try the various exponents up to this bound. Note that  this is far from efficient. 

Is there a better bound on the  torsion?  
 
 
This 1997  \href{https://www.ams.org/journals/proc/1997-125-12/S0002-9939-97-04283-4/S0002-9939-97-04283-4.pdf}{paper by Shmuel Friedland} shows that  up to conjugation, the  group  O(d) of orthogonal $d \times d$ matrices over the integers is the maximal finite subgroup of $GL_d(\QQ)$,  for large d. Also, its order $2^d d!$  is the maximal order of a finite subgroup,  with the exception of $d= 2,4,6,7,8,9,10$.  In the exceptional  cases, the maximum has been obtained by Feit, based on Weisfeiler's tight estimate for the Jordan function, in a  preprint called ``The orders of finite linear groups",  1995, which has not   appeared.  See these \href{https://www.college-de-france.fr/media/jean-pierre-serre/UPL3821667391778701726_6___Bounds_for_the_orders.pdf}{2006 notes by Serre}, where the work of Weisfeiler-Feit is called ``incomplete".
 
 If it works,  then for $d =3,5 $ or $d > 10$,  the maximum order is   $2L(d)$, where $L(d)$ is the maximum order of an element of the  symmetric group $S_d$. (This follows because $O(d) \cong  C_2^d \rtimes S(d)$.)  For the value of $L(d) $,  see this \href{https://www.jstor.org/stable/2322839?seq=1#metadata_info_tab_contents}{paper by William Miller}. It surveys= a proof that 
\bc  $\log  L(d) \sim \sqrt {d \log d}$, \ec where $\log$ is the natural logarithm. Very roughly, this says that   $L(d)$ is  the square root of the size of $S(d)$, which is about $\exp{(d \log d)}$.

 \def\C{{\mathrm C}}
 \section{Khelif,Wilson: On ultraproducts of alternating groups}
  Gabriel Sabbagh asked  whether an infinite pseudo-finite group can be finitely generated. For some background  see \cite[Section~6]{LogicBlog:20}. 
  
  This section  is based on notes by Anatole Khelif related to this problem. John Wilson and Nies discussed these by email. The   full explication below is due to   Wilson. 
 Note that pseudofinite groups are precisely the groups elementarily equivalent to an ultraproduct of finite groups.  The main result~is: 
\begin{theorem} No infinite finitely generated group is elementarily equivalent to an ultraproduct of a family of finite alternating groups. \end{theorem}

To prove this, 
we require two lemmata.  
 For the definition and properties of bi-interpretability we refer the reader to \cite[p.\ 222] {Hodges:93}.   We write $P_n$ for the power set of the set $\{1,\dots,n\}$.

\begin{lemma}
  For $n \geq 9$, the alternating group $A_n$ and the three-sorted structure 
$(A_n , \{1\dots,n\}, P_n)$, with the group operations on $A_n$, the action of $A_n$ on $ \{1\dots,n\}$, the element relation between $\{1\dots,n\}$ and sets in $ P_n$ and the Boolean algebra operations on $P_n$, are bi-interpretable,  in a uniformly way, without parameters.   \end{lemma}

\begin{proof}  It suffices to prove that $(A_n , \{1\dots,n\}, P_n)$ is uniformly interpretable in $A_n$.

For a subset $X$ of $\{1,2,\dots,n\}$ write $A_X$ for the subgroup of $A_n$ that acts as the alternating group on $X$ and fixes the complement $X'$ pointwise.  For $u\in A_n$ write $B_u$ for the support $\{i\in\{1,\dots,n\}\mid ui\neq i\}$ of $u$ and $C_u$ for the centralizer $\C_{A_n}(u)$ of $u$ in $A_n$.  We write $C_u'$ for the derived group (commutator subgroup) of $C_u$, and $C_u''$ for the derived subgroup of $C_u'$.  

We say that $u$ has the $V$-property if $u$ has order $2$ and exactly four elements of $\C_u$ commute with all of their conjugates under elements of $C_u$.  We claim that $u$ has the $V$-property if and only $u$ has order $2$ and $|B_u|=4$, and that if this holds, then $C_u''=A_{B_u'}$.  Since the elements $u$ of order $2$ are products of disjoint transpositions and are determined up to conjugacy by $B_u$, it suffices to consider permutations $u=(1,2)(3,4)\dots(4k-1,4k)$ for positive integers $k$.  An element $A_n$
lies in $C_u$ if and only if it permutes the sets $\{1,2\},\{3,4\},\dots,\{4k- 1,4k\}$.  Therefore each product of an even number of the transpositions $(1,2),(3,4),\dots,(4k-1,4k)$ lies in $C_u$ and commutes with all conjugates by elements $C_u$.  Since there are $k(2k-1)$ products with two factors,  if $u$ has the $V$-property then $k=1$ and $|B_u|=4$.  
    
Now let $u=(1,2)(3,4)$.
Let $D$ be the dihedral subgroup $\langle(1,3,2,4),(1,2)\rangle$ of order $8$ in $S_n$, and $S$ the subgroup of $S_n$ fixing $\{1,2,3,4\}$ pointwise.  Then $S\cong S_{n-4}$ and, since $\C_{S_n}(u)$ maps $\{1,2,3,4\}$ to itself, $C_u$ is the subgroup of $D\times S$ consisting of pairs $xy$ with $x\in D,y\in S$ and $x,y$ both even or both odd.  
Since $S_{n-4}$ has no non-trivial abelian normal subgroups, consideration of the projection map from $C_u$ to $S$ shows that the elements of $C_u$ commuting with all conjugates in $C_u$ lie in $D$, and, being even, they are the elements of $\langle(1,2)(3,4),(13),(2,4)\rangle$.  Thus $|B_u|=4$ and $u$ has the $V$-property.  Moreover, $C_u'=D'\times A_{\{1,2,3,4\}'}$ and the second derived group $C_u''$ of $C_u$ is $A_{\{1,2,3,4\}'}$.  Our claim follows.  
We note also that, since every element of $A_{n-4}$ is a commutator, $C_u''$ is the set of iterated commutators $[[x_1,x_2],[x_3,x_4]]$ with
$x_1,\dots,x_4$ in $C_u$ and so $C_u''$ is definable in $A_n$ with parameter $u$, uniformly in $n$. 

Let $u,v$ be elements of $A_n$ with the $V$-property.  From above, $\langle C_u'' , C_v''\rangle$ fixes the
elements of $B_u \cap B_v$ and indeed it easy to check that it is exactly the subgroup of $A_n$ fixing the elements of $B_u \cap B_v$; that is,
$$\langle C_u'' , C_v''\rangle=A_{(B_u\cap B_v)'}.$$    
We claim next that the above subgroup is definable with parameters $u,v$ uniformly in $n$.  From above we can assume that $B_u\neq B_v$,
and therefore that $|B_u\cap B_v|\leq3$. 
For $H\subseteq A_n$ write $\C(H)$ for $\C_{A_n}(H)$.  Note that if $X\in P_n$ and $|X|\geq 5$ then $\C(A_X)=A_{X'}$.  Now
\vskip-.15in 
$$\C(A_{B_u'})\cap \C(A_{B_v'})=A_{B_u}\cap  A_{B_v}=A_{B_u\cap B_v}.$$ So 
$\C(\C(A_{B_u'})\cap \C(A_{B_v'}))=\C(A_{B_u\cap B_v})$ contains $A_{(B_u\cap B_v)'}$ and is contained in the direct product of the symmetric group on $(B_u\cap B_v)'$ and a group of permutations of $B_u\cap B_v$.   Since $S_3$ is soluble of derived length $2$ and for $m\geq5$ every element of $A_m$ is a commutator, it follows that \vskip-.15in 
$$A_{(B_u\cap B_v)'}=\{[[x_1,x_2],[x_3,x_4]]\mid x_1,\dots,x_4\in\C(\C(A_{B_u'})\cap \C(A_{B_v'}))\}.$$
Our claim follows.

The maximal proper subgroups of $A_n$ of the form $A_{(B_u\cap B_v)'}$ are exactly the point stabilizers in $A_n$. They can be identified with the corresponding elements of $\{1,\dots,n\}$ and then the conjugation action of $A_n$ on the stabilizers becomes the natural action of $A_n$ on $\{1,\dots,n\}$. The elements of $P(\{1,\dots,n\})$ can be defined as the fixed-point sets of elements of $A_n$ or, in the case of subsets with complement containing at most two elements, the union of two such sets.
The lemma follows.
\end{proof}

\begin{lemma} \label{lem:second}  Let $G$ be a finitely generated group and let $L$ be an atomic infinite Boolean algebra with least element $0$ on which $G$ acts as lattice automorphisms. Assume that the restriction of the action to the set of atoms is transitive and that for all $a,b \in L$, there is $g \in G$ such that
  $ga \leq b$ or $gb\leq a$.
Then $(\NN,\leq)$ is interpretable in $(G,L)$ with parameters (the parameters being a finite generating set for $G$). \end{lemma}

\begin{proof}
Let $S$ be a finite symmetric generating subset of $G$.
We say that an element $a\in L$ is {\em connected} if for all $c,d\neq0$ with $c \wedge d = 0$
 and $c \vee d = a$ there are atoms $u$, $v$ and $s \in S$ with the following property: $u \leq c$, $v \leq d$ and $sc = d$.

We define a preorder on $L$ as follows: for $a,b\in L$ write
$a\preceq b$ if and only if there exists $g\in G$ such that
$ga \leq b$.  Then $\preceq$ induces an equivalence relation $\sim$ on $L$ and a total order $\leq$ on the quotient set $L/{\sim}$: for $a,b\in L$ we have
$a\sim b$ if and only if $a\preceq b$ and $b\preceq a$; and for equivalence classes ${\bf c}, {\bf d}$ we have ${\bf c}\leq {\bf d}$ if and only if there are $c\in \bf c$ and $d\in \bf d$ such that $c\leq d$.

Let $s$ and $t$ be atoms of $L$.
Since $G$ acts transitively on the atoms, there exists a connected element $a\in L$
containing finitely many atoms such that $s \leq a$ and $t \leq a$.
Then we can consider a minimal (for $\leq$) connected element $a\in L$ such that $s \leq a$ and $t \leq a$.
Such an element can be regarded as a finite sequence $a_1, \dots, a_n$ such that for any $i\in
\{1,\dots,n-1\}$, there is $s_i \in S$ such that $s_i(a_i) = a_{i+1}$.  Since $G$ is finitely generated and acts transitively on the (infinitely many) atoms of $L$, such minimal sequences can be arbitrarily long.

We claim next that if $u,v$ are elements of $L$ that  contain only finitely many atoms, then they are equivalent if and only if they have the same number of atoms.
Without loss of generality, we can suppose that $gu \leq v$ for some $g\in G$. It is obvious that $v$ contains at least as many atoms as $u$.
If $u$ and $v$ are equivalent, by symmetry we deduce that they contain the same number of atoms.
Conversely, if $u$ and $v$ contain the same (finite) number of atoms then $v \setminus gu$ does not contain any atom. Since $L$ is atomic, it follows that $gu = v$. Our claim follows.

We conclude that the set of equivalence classes of such minimal connected elements with the order induced by $\leq$ is isomorphic to $(\NN,\leq)$. 
\end{proof}

\begin{proof}[Proof of the theorem]
We suppose that the theorem is false.
Let $G$ be a group elementarily equivalent to an infinite ultraproduct of alternating groups $A_n$; then the integers $n$ are unbounded on each set in the underlying ultrafilter.  By Lemma 1, this group acts transitively on the atoms of an infinite atomic $G$-definable Boolean algebra (that is elementarily equivalent to an ultraproduct of sets $P_n$ with $n$ unbounded on sets in the ultrafilter).  
 
By Lemma 2, if $G$ is also finitely generated then $(\NN,\leq)$ is interpretable in $G$.
 Since  $(\NN,\leq)$   is a finitely axiomatisable infinite
structure,  a pseudofinite group cannot interpret it. This is a contradiction.
\end{proof}

\section{Khelif, Nies: f.g.\ groups acting on Boolean algebras}

This section contains another version of the second lemma in Khelif's notes, already discussed as Lemma~\ref{lem:second} above. The proof of this version is explicated by Nies.
 
\begin{lemma}  \label{prop:Khelif} Let $G$ be a f.g.\ group acting from the left on an infinite atomic Boolean algebra $A$,  canonically identified with a subalgebra of the powerset of the set of atoms. Suppose that for each $n \in \NN$, the action is transitive for elements  of   cardinality $n$.

Then $(\NN, \le)$ can be interpreted in $(G,A)$ with parameters.

In particular, $G$  fails to be  pseudofinite. \end{lemma}

\begin{proof} Let $S $ be a finite symmetric generating set of $G$. We use letters $a,b,c,d$ for elements of $A$, and $u,v,w$ for atoms. We write \bc $a \preceq b \LR \ex g \in G \, [g\cdot a \le b]$.  \ec
If $a$ is finite then by the transitivity hypothesis, 
\bc $|a| \le |b| \LR a \preceq b$. \ec 

We say that $a$ is \emph{connected} if for each $c$ such that $0< c < a$, there are $s \in S$ and  and atom $u \le c$ such that $su \wedge c = 0$. 
Given $a \in A$, let $\Gamma_a$ be the graph on the elements of $a$ (i.e., atoms below $a$) with the edge relation 
\bc $Euv \LR \ex s \in S [ s\cdot u= v] $. \ec Note that the    graph $\Gamma_1$, where $1$ is the largest element of $A$,  is connected since $G$ is generated by $S$ and the action on the atoms is transitive. 

\begin{claim} Let $a$ be finite. Then $a$ is connected $\LR$ $\Gamma_a$ is connected. \end{claim}

For the implication  right to left: one can choose atoms $v \le c$ and $w \le a \setminus c$. There is a path from $v$ to $w$, which at some point leaves $c$. 

For the   implication left to right: if $\Gamma_a$ is not connected, there is an atom $v\le a$ such that its connected component does not contain some atom $w$. Since $a$ is finite, its connected component $c$ exists in $A$, and satisfies $c< a$. But there can be no atom $u \le c$ as in the definition of an element being connected.

Let $a$ be connected and suppose that $u,v \le a$ are atoms. We say that $a$ is \emph{minimal connected for $u,v$} if $a \preceq b$ holds for any connected $b$ such that $u,v \le b$. 
\begin{claim} Let $u,v$ be atoms. Let $d $ be the set of atoms in a path   between  $u$ and $v$  of shortest length in the whole graph $\Gamma_1$. 
\bi \item[(a)] The set $d$ is minimal connected for $u,v$.
\item[(b)] If $a$ is minimal connected for $u,v$ then $|d| = |a|$.  \ei
\end{claim}

For (a), note that clearly $d$ is connected. If $|c| < |d|$ and $u,v \in c$ then by Claim 2 and the assumption that the path was shortest,  $c$ is not connected for $u,v$. Hence for any connected $c$ such that $u,v \in c$, we have $d \preceq c$ because $c$ contains at least $|d|$ many atoms. 

For (b), note that by definition we have $a \preceq d$ and hence  $|a| \le |d|$. By (a) this implies that $|a| = |d|$.  This shows the claim. 

Now let $\+ F$ be the set of $a$ that are minimal connected for some atoms $u,v$. Then $\+ F$ contains only finite sets, and contains sets of each finite  cardinality because $A$ is infinite. So the quotient partial order of the pre-order $( \+ F, \preceq)$, which can be interpreted in $(G,A)$ with parameters from $S$,  is isomorphic to $(\NN, \le)$.  
\end{proof}

\section{Nies: The meet groupoid of an oligomorphic group}
\begin{definition} 
\label{def:MeetGroupoid} A \emph{meet groupoid} as defined in~\cite{Melnikov.Nies:22} is a groupoid  $(H, \cdot , {}^{-1})$ that is also a meet semilattice  $(H, \wedge  ,\ES)$ with least element $\ES$, and 
  satisfies the following conditions, where  $\sqsubseteq$ denotes the induced partial order  given by $A\sqsubseteq B :\Leftrightarrow A\wedge B=A$.

\begin{enumerate} 
\item 
(least element) 
$\ES$ is an idempotent such that $\ES \cdot A$ and $A \cdot \ES$ are only defined for $A = \ES$. 

\item
(compatibility)
$U  \wedge V \neq \ES$ for  all idempotents $U,V\neq\emptyset$. 

\item 
(monotonicity of inverse operation)
$A\sqsubseteq B$ $\Leftrightarrow A^{-1} \sqsubseteq B^{-1}$. 

\item 
If $A_i \cdot B_i$ are defined for $i=0,1$ and $A_0\cap A_1\neq \emptyset \neq B_0\cap B_1$, then 
$$ (A_0\cap A_1)  \cdot (B_0\cap B_1) = (A_0\cdot B_0) \cap (A_1 \cdot B_1). $$ 
\end{enumerate} 
\end{definition} 

 In the following let $G \le_c \S$ be an oligomorphic closed subgroup of the symmetric group $\S$.
  The meet groupoid $\+ W(G)$ of $G$ is   a variant of the coarse group $\+ M(G)$ of $G$  introduced in \cite{Nies.Schlicht.etal:21}, which consists of the open cosets with a ternary relation for $AB \subseteq C$.  The structure $\+ W(G)$  was briefly introduced in \cite[Remark 2.16]{Nies.Schlicht.etal:21}, where it was noted that  $\+ M(G)$   and $\+ W(G)$ are interdefinable in a uniform way.  
\begin{definition} \label{def:meetgr} Given an oligomorphic group $G$, the domain of the \emph{meet groupoid} $\+ W(G)$   consists of the open cosets, together with $\emptyset$. The groupoid product of open cosets $A \cdot B$ is defined iff $A$ is a left coset of a subgroup that $B$ is a right coset of. The meet is the usual intersection as a binary operation.  \end{definition}

 \subsection{A model theoretic view of $\+ W(G)$}
 
 We show that arbitrarily large canonical pieces of $\+ W(G)$ can be interpreted in $M_G$, where $M_G$ is the canonical structure for $G$.
  \subsubsection*{The structure of imaginaries for an $\omega$-categorical structure $M$}
 Let $M$ be a countable 
structure with domain $\omega$. By ``definable" we mean first-order definable in $M$  without parameters. The structure $M^{eq}$ of imaginaries over~$M$, going back to Poizat and Shelah,   has sorts $S= D/E$ for any $k \ge 1$,  nonempty  definable  $D \sub M^k$, and definable equivalence relation $E$ on $D$. We will use $\aaa, \beta, \ldots$ for elements of $M^{eq}$.  The atomic   relations of $M^{eq}$  are the  atomic  relations of the given structure $M$ (viewed as relations on the  ``home sort" $M$), as well as the graphs of the projections $D \to S=D/E$ for each sort $S$. 
 
 In a many-sorted structure, each definable relation is a subset of $\prod_{i< n} S_i$ for sorts  $S_0, \ldots, S_{n-1}$.  The   definable relations are given  as usual via   formulas in the   first-order logic in the corresponding signature  with variables for each sort.  Extending the usual  definition of algebraic/definable closure to a many-sorted structure is no problem. We note that assuming $M$ is $\omega$-categorical, $\text{acl}(A)  \cap S$ is finite for each finite $A \sub M^{eq}$ and  each sort $S$. 
 
 \subsubsection*{Open subgroups as stabilisers of imaginaries}
  Let $ G\le_c \S$ be oligomorphic. It is well-known that   $G= \Aut(M_G)$ for an $\omega$-categorical relational structure $M_G$ that for each $k$, for    each $k$-orbit of $G$ has a named $k$-ary relation. In this context one says that $M_G$ is a canonical structure for $G$; it is unique up to renaming the $k$-orbits, for each $k$.  Letting $G$ act on $M_G^{eq}$ in the canonical way, we have  the following well known fact  due to Ahlbrandt/Ziegler \cite{Ahlbrandt.Ziegler:86}  (see Theorem 2.7 in   \href{http://wwwf.imperial.ac.uk/~dmevans/Bonn2013_DE.pdf}{ David Evans' 2013  notes}).
  $U \le_o G$ denotes that $U$ is an open subgroup of $G$.

\begin{fact} \label{fac1}  $U \le_o G$ if and only if $ U = G_{\aaa}$ for some $\aaa \in M_G^{eq}$.

\n Moreover, for any $\bar{a}$ with $G_{\bar{a}}\subseteq U$, $\alpha$ can be chosen to be of the form $\bar{a}/E$, where $E$ is some definable 
equivalence relation. 
\end{fact}
 
 \begin{proof} 
 For the implication ``$\LA$" and the additional assertion, suppose that $\aaa \in D/E$ as above. We have $\aaa = \ol a  /E$ for some $\ol a \in D$. The pointwise stabiliser $G_{\ol a}$ is open by definition of the topology of $G$, and is contained in $G_\aaa$. So $G_\aaa$ is open as well.

  For the implication ``$\RA$",  pick $\ol a$ such that  $G_{\ol a}\sub U$. Let $D = G\cdot \ol a$, and define an  equivalence relation $E$ on $D$ by 
  \[  h \ol a E g  \ol a  \lra g^{-1}h \in U.\]
  Let $\alpha$ be the equivalence class of $\ol a$. 
  Then we have $G_\aaa = U$, because \\ $g(\aaa) = \aaa \lra g(\ol a ) E \ol a \lra g \in U$. 
\end{proof} 

\begin{definition} Let $D$ be a set of $n$-tuples and $E$ an equivalence relation on~$D$. The relation $E$ is called \emph{degenerate} if for some   set $\sigma \subsetneq n$, the equivalence relation $E_{\sss,D}= \{\langle \bar{a}, \bar{b}\rangle  \in D^2 \colon \forall i \in \sigma \, \bar a_i=\bar b_i\}$ is contained in $E$. Otherwise, call $E$ non-degenerate.
\end{definition}

\begin{fact} If  $\ol a$ is chosen of minimal length such that $G_{\ol a} \le U$, then the equivalence relation $E$ is non-degenerate. \end{fact}
\begin{proof} Assume for contradiction that  it is degenerate via $\sigma$. We may assume that $\sigma $ consists of the numbers $0, \ldots, n-2$. Let $\ol a' = \ol a \uhr{n-1}$.  We claim that $G_{\ol a'} \le U$. For let $h\in G_{\ol a'}$. 
Then $h(\ol a) E_\sigma \ol a$ and hence $h(\ol a) E \ol a$ holds,     so that $h \in U$.
\end{proof}

Given an open subgroup $G_{\alpha}$, note that the orbit equivalence relation $E_{G_{\alpha}}$ on $M^k$   is definable from $\alpha$, since $\bar{c} E_{G_{\alpha}} \bar{d}$ if and only if for some $\bar{a}, \bar{b}\in\alpha$, $\bar{a}^\smallfrown \bar{c}$ and $\bar{b}^\smallfrown \bar{d}$ have the same type. 

 The next fact states a model theoretic characterisation   of the containment relation between open  subgroups of $G$, and draws   a surprising consequence for the lattice of open subgroups.

  \begin{fact}  \label{fac2}\

 \bi  \item[(i)] For $\aaa, \beta \in M^{\mrm{eq}}$, one has $G_\aaa \sub G_\beta \LR \beta \in \text{dcl}(\aaa)$.

\item[(ii)] Each open  subgroup $U$ is contained in  only finitely many open subgroups.
\ei  
 \end{fact}
 \begin{proof}
 (i) The implication ``$\LA$'' is immediate. 
 
 For ``$\RA$", suppose $\aaa = \ol a /E$ and $\beta = \wt b /F$ as above. Then 
 \bc $h \ol a  E g \ol a  \lra g^{-1}h \in G_\alpha$  and   $h \wt b  F g \wt b  \lra g^{-1}h \in G_\beta$.  \ec
 Let $R=  G \cdot \la \ol a , \wt b\ra$, and note that $R$ is a left invariant, and hence  definable, relation. 
 Since $G_\alpha \sub G_\beta$, 
 $\beta$ is  the unique equivalence class of $F$ containing $\wt  b '$ for some $\langle  \bar a',\wt b'\rangle \in R$ with $\bar a '\in \alpha$. 
Hence $\beta \in \text{dcl}(\aaa)$. 
 
\n  (ii) If $U \sub V$ where $U = G_\aaa$, $\aaa = \ol a/E$,   by \cref{fac1}  we have  $V = G_\beta $ for some imaginary $\beta = \ol a /F$, and $E \sub F$. So the number of open subgroups containing $U$ is bounded by the number of definable equivalence relations on $G\cdot \ol a$ containing $E$.  
 \end{proof}

 There is a natural generalisation of (i), proved in the same fashion. For $\aaa_1, \ldots, \aaa_k, \beta \in M^{eq}$, one has \bc $\bigcap_iG_{\aaa_i} \sub G_\beta \LR \beta \in \text{dcl}(\aaa_1, \ldots, \aaa_l)$. \ec

%
%
%
%
%
%
%
%

 \subsubsection*{Open cosets}
 
 For imaginaries $\aaa_0, \aaa_1$ of the same sort, let 
 \[[\aaa_0, \aaa_1] := G_{\aaa_0 , \aaa_1}: = \{ g \in G \colon g(\aaa_1)= \aaa_0\}. \] 
 In particular, $G_{\aaa, \aaa} = G_\aaa$. Clearly if $G_{\aaa_0 , \aaa_1}\neq \ES$ then it  is a right coset of $G_{\aaa_0}$,  and is a left coset of $G_{\aaa_1}$. Also $G_{\aaa_1 , \aaa_0}= G_{\aaa_0 , \aaa_1}^{-1}$. 
 
 \begin{fact} \label{coset and pairs} Each open coset $A$ is of the form $[\aaa_0, \aaa_1]$. Here $\aaa_1$ can be chosen to be any imaginary such that $A$ is a left coset of $G_{\aaa_1}$.
  \end{fact}
 
 \begin{proof} Write $A = h G_{\aaa_1}$ for some $h \in G$,  and let $\aaa_0 = h(\aaa_1)$. \end{proof}

 We now extend \cref{fac2}(i) to general open cosets. 
 
   \begin{fact}  \label{fac3}
  
       $G_{\aaa_0 , \aaa_1}=:A  \sub B:= G_{\beta_0 , \beta_1}  \LR $
       
       \hfill $A = \ES \vee \beta_i \in \text{dcl}(\aaa_i)$ ($i=0,1)$ via the same formula. \end{fact}
  \begin{proof}
 (i) The implication ``$\LA$'' is again  immediate. 
 
 \n For ``$\RA$", suppose $\aaa_i = \ol a_i /E$ and $\beta_i = \wt  b_i /F$ as above, for $i=0,1$. Since $G_{\aaa_1} = A^{-1} A \sub B^{-1} B=     G_{\beta_1}$,   we have $\beta_1 \in \text{dcl}(\aaa_1)$,   via a formula $\phi(\ol x, \wt y)$. 
 We may assume $A \neq \ES$. Let  $g\in A$, so that  $g (\aaa_1) = \aaa_0$. Then  $g (\beta_1) = \beta_0$. Hence $\beta_0 \in \text{dcl}(\aaa_0)$ via~$\phi$. \end{proof}

\begin{fact}  \label{count cosets} Let $V$ be an open subgroup. Let $k= |\mrm{dcl}(\beta) \cap S|$, where $V = G_\beta$ and $\beta $ is of sort $S$ such that  $S$ is a single orbit in $M^{\mrm{eq}}$; note that by \ref{fac2}(i),  $k$  does not depend on the choice of $S$. For each open subgroup $U$,   the  number of left  cosets $A$ of $V$ that are right cosets of $U$ is   either $0$  or $k$. \end{fact}

\begin{proof} 
If $U$ is not conjugate to $V$ then the number of such cosets is $0$.  Otherwise let  $A$ be  such a coset:  by~\cref{coset and pairs} we may write $A = [\aaa, \beta]$ for some $\aaa \in S$ such that $U = G_\aaa$. Any coset $A'$ with    $_{U}A' _V$ can be written as $[\aaa', \beta]$ where  $G_{\aaa'}= U$. If $A''=[\aaa'', \beta]$ is another such coset, then $A'= A''$ iff $\aaa' = \aaa''$. So the number of cosets is $ |\mrm{dcl}(\aaa) \cap S|$,  which equals $k$ as $\aaa$ and $\beta $ are in the same $G$-orbit. 
 \end{proof}
 
 \begin{example} 
 Let $k \ge 2$. Let $M$ be the structure with one permutation function   $\pi$, where  each cycle   has length $k$. Then $G= C_k \wr \S$ where the wreath product is taken via the action of $\S$ on $\omega$; this acts 1-transitively on $M$. Let $S$ be the home sort. Let $V$ be the stabiliser of some element $b$. For each $U$ stabiliser of some $a$ there are $k$ cosets $A_i$ $(i< k)$ with $_{U}A _V$: we have $A_i = \{ f \colon f(b) = \pi^{i}(a)\}$. 
 \end{example}
 
 \subsubsection*{Meet groupoid operations}

We   show that the meet groupoid  operations on $\+ W(G)$ are definable in $M^{\mrm{eq}}$ when restricted to particular sorts. Recall that we write $[\aaa_0, \aaa_1] := G_{\aaa_0 , \aaa_1}$. Also note that since $M$ is $\omega$-categorical, interdefinability between elements of two fixed sorts is definable in $M$: up to equivalence in $\mrm{Th}(M)$ there are only finitely many possibilities of pairs of formulas witnessing  the interdefinability.

For the  \emph{intersection} operation, 
suppose $S, T $ are sorts, and $\aaa_i \in S$ and $\beta_i \in T$ for $i=0,1$. Let $\gamma_i = \langle \alpha_i, \beta_i \rangle  \in S \times T$. Clearly  $[\aaa_0 , \aaa_1] \cap [\beta_0 , \beta_1] = [\gamma_0 , \gamma_1]$. So given a sort $R$ and $\delta_0, \delta_1 \in R$, by~\ref{fac3} the 6-ary relation 
  $[\aaa_0 , \aaa_1] \cap [\beta_0 , \beta_1]= [\delta_0 , \delta_1]$ is definable in $M^{\mrm{eq}}$.

The  {\it 
  inverse} is definable because $[\aaa_0 , \aaa_1] ^{-1} = [\aaa_1 , \aaa_0]$.

This leaves the  \emph{groupoid product} of two  cosets:
\begin{fact} Given sorts  $R,S,T$ of $M^{\mrm{eq}}$, 
the  following relation is definable: \bc $[\aaa_0, \aaa_1] \cdot [\beta_0, \beta_1]= [\gamma_0, \gamma_1]$, where $\aaa_i \in R, \beta_i\in S, \gamma_i \in T$. \ec \end{fact}
 \begin{proof} Note that $[\aaa_0, \aaa_1] \cdot [\beta_0, \beta_1]= [\gamma_0, \gamma_1]$ iff 

 \emph{$\aaa_1, \beta_0$ are interdefinable via some formulas $\phi, \psi$, and
 $\exists \aaa_0' \in S$ 
 \\
 $[ \aaa_0, \aaa_0'$ interdefinable via the same formulas $\phi, \psi$, and $[ \aaa'_0 , \beta_1] = [\gamma_0, \gamma_1]$. }

The implication from right to left  holds  because $[\aaa_0', \beta_0] \cdot [\beta_0, \beta_1]= [ \aaa'_0 , \beta_1]$, and by~\ref{fac3}. The converse implication holds because $[\aaa_0, \aaa_1] \cdot [\beta_0, \beta_1]$ is defined iff  $G_{\aaa_1}= G_{\beta_0}$. So $\aaa_1, \beta_0$ are interdefinable by \ref{fac2}. 

This suffices because there are only finitely many possibilities for $\phi, \psi$ (up to equivalence relative to  the theory of $M$).
\end{proof}

  \section{Nies: The   quotient of $N_G$ by $C_G$ for $1$-transitive oligomorphic $G$} 
  
 For an oligomorphic group $G$, the centre $Z(G)$ is finite because each $g\in Z(G)$ is $G$-invariant as a binary relation.   In  \cite[Th.\ 4.5]{LogicBlog:18} it was proved    that $G/Z(G)$ is homeomorphic to an oligomorphic group. The following is a variation on  this result, with a   proof   similar to the original one.  
 
 For a countable structure $M$ let  $\+ E_M$ is the orbital structure of $M$. It has a $2n$-ary relation symbol for each $n$, that two  $n$-tuples are  in the same orbit.  Note that if $G = \Aut(M)$ for $\omega$-categorical $M$, then the normaliser $N_G$ is $\Aut(\+ E_M)$.
 
   \begin{prop} \label{prop:CG} Let $G \le_c \S$ be   oligomorphic and 1-transitive. Let $N_G$ be the normaliser of $G$ in $\S$, and let $C_G$ be its centraliser. Then $N_G/C_G$ is  homeomorphic to an oligomorphic group. 
  \end{prop}
  The hypothesis that $G$ is 1-transitive is for simplicity only. Note here that both $N_G$ and $C_G$ are closed subgroups of $\S$, and that $C_G$ is a normal subgroup of $N_G$.  So $H= N_G/C_G$ is   a Polish group via  the quotient topology:  for $C_G \le U \le N_G$, the   subgroup $U/C \le H$ is declared to be open iff $U$ is open in $N_G$.  (See    \cite[Prop.\ 2.2.10]{Gao:09}.) Also $C_G$ is finite because each of its elements is  $G$-invariant when seen as a binary relation on $\omega$, and hence definable in the canonical ($\aleph_0$-categorical) structure for $G$. The group   $H$   naturally embeds  $G/Z(G)$ as a closed subgroup because $G \cap C_G = Z(G)$. However $C_G$ can be much bigger than $Z(G)$. In the finite setting, simply view a nonabelian group $G$ on $n$ elements  as a subgroup of $S_n$ via  the left regular presentation. Then $C_G$ contains the right regular presentation, and thus an isomorphic copy of $G$. For an oligomorphic  example building on this, take a wreath product by $\S$.

\begin{remark} \label{rmk:action} Suppose $\delta \colon L \times X \to X$ is a left action of a group $L$   on a set $X$, and $C \trianglelefteq L$. Write $\sim_C$ for the orbit equivalence relation of the subaction of $C$. Then $L$ acts from the left   on $X/\sim_C$ via $\delta_C (g,  [x] )= [\delta(g,  x)]$ (where $[x]$ is the $\sim_C$ class of $x$).  Since the elements of $C$ act as the identity, $L/C$ also acts on $X / \sim_C$.  \end{remark}
 
 \begin{proof}[Proof of Prop.~\ref{prop:CG}]
 In the notation above, let $X =  \NN \times \NN$ and let  $N_G$ act on $X$ by the usual diagonal action. Write  $C= C_G$.  
 
 \begin{claim} The action of $N_G /C$ on $X / \sim_{C}$ is faithful. \end{claim}
\n  To see this, suppose $g \in N_G \setminus C$. Let $\eta \in G$ and $w \in \NN$ be such that \bc $ (g \eta) \cdot w \neq (\eta g) \cdot w$. \ec  Let $w'= \eta \cdot w$.  Then $ g \cdot (w,w') \not \sim_C (w,w')$,  because any permutation $h $  such that $h \cdot (w,w') = g \cdot (w,w')$ satisfies $(h \eta) \cdot w \neq (\eta h) \cdot w$, so that $h \not \in C$.  This shows the claim.

 Given  a Polish group $H$  with a faithful  action $\gamma \colon H \times V \to V$ on  a  countable set $V$,  we obtain a monomorphism $\Theta_\gamma \colon H \to \mathrm{Sym}(V)$ given by $\Theta_\gamma(g)(k)= \gamma(g,k)$. A Polish group action is continuous iff it is separately continuous (i.e. when one argument is fixed). In the case of an action on countable $V$ (with the discrete topology), the latter condition  means that  
 
 \bi \item[(a)] for each $k,n \in V$, the set $\{ g \colon \gamma(g,k)=n\}$ is open. \ei

 So $\gamma$ is continuous iff $\Theta_\gamma$ is continuous. 

\begin{definition} 
We say that a faithful  action $\gamma \colon H \times V \to V$   is \emph{strongly continuous} if the embedding $\Theta_\gamma$ is topological. 
\end{definition} 
\n Equivalently, the action is continuous,  and for each open set $U$ containing $1_H$, also $\Theta_\gamma(U)$ is open, namely,   

 \bi \item[(b)]
for each open  $U$ containing $1_H$, there is a  finite set $B \sub V$  such that $\fa k \in B \, [\gamma(g,k) = k]$ implies $g \in U$.  \ei Since $H$ is Polish, strong continuity of the action  implies that $H$  is topologically isomorphic  via $\Theta_\gamma$ to a closed subgroup of $\mathrm{Sym}(V)$ (see e.g.\  \cite[Prop.\ 2.2.1]{Gao:09}).

  Let  $\delta$ be the canonical action of $N_G$  on $\NN \times \NN$, and note that it has finitely many orbits. Now let $V:= \NN \times \NN/\sim_{C}$.  By Remark~\ref{rmk:action}, we obtain an action  $\gamma$ of $N_G/C_G$ on $V$, which by the above yields   an (algebraic) embedding $\Theta_\gamma$ of $N_G/C$ into $\mathrm{Sym}(V)$. 
 
 \begin{claim}  \label{cl:1tr}  $\Theta_\gamma \colon N_G/C_G \to \mathrm{Sym}(V)$  is a topological embedding. \end{claim}
To see this, 
 we check the conditions (a) and (b) above.  
 
 \medskip 
 
\n  (a) Suppose that $k = [(w_0, w'_0)], n  = [(w_1, w'_1)]$. Then $\gamma(g,k)=n$ iff there are $(v_0, v'_0) \sim_C (w_0, w'_0)$ and  $(v_1, v'_1) \sim_C (w_1, w'_1)$ such that $g \cdot v_0 = v_1$ and $g \cdot v'_0 = v'_1$. The set of such $g$ is  open in $N_G$.

 \medskip 
 \n  (b) An open  neighbourhood   of  $1_{N_G/C}$ has the form $U/C$ wherewhere $U \sub N_G$ is   open   and $C \le U$.   By definition of the topology on $N_G$, we may assume that  $U = (N_G)_{R} C$ for some finite $R \sub \NN$ (as usual $(N_G)_R$ is the pointwise stabiliser of the set $R$). Let $B = (R \times R) / \sim_C$. Let  $g = pC \in H$, where $p \in N_G$. 
 
 Suppose that  $  \gamma(g,k) = k$ for each $k \in B$. This means that for each $u,v \in R$, there is a $c_{u,v} \in C$ such that $p \cdot  (u,v) = c_{u,v} \cdot  (u,v)$. Since $G$ acts faithfully and 1-transitively on $\NN$, for each $c,d \in C$, if there is  $y \in \NN$ such that  $c\cdot  y = d \cdot y$, then  $c=d$. Therefore, given another pair $r,s \in R$, we have $c_{u,v} = c_{u,s} = c_{r,s}$. Let $c \in C$ be this unique witness; then $p \cdot u = c \cdot u$ for each $u \in R$, hence  $p \in (N_G)_R C$ and therefore $g \in U/C$.
 This shows the claim.
 
 
 Since $N_G/C_G$ is Roelcke precompact and its action $\gamma$ on $V$
  has only finitely many 1-orbits, by   \cite[Theorem 2.4]{Tsankov:12} the action $\gamma$ is oligomorphic.   \end{proof}

 \renewcommand{\S} {\Sym(\NN)}

  \section{Nies:  $\Aut(G)$ and $\Out(G)$   for   closed  subgroups $G$ of $\S$}
  
 Most results in this post have been obtained in co-work with either Gianluca Paolini~\cite{Nies.Paolini:24}, or Philipp Schlicht. By $\Aut(G)$ we will denote the group of (bi-)\emph{continuous} automorphisms of  a Polish group~$G$. 
 {We discuss how to put a natural Polish  topology on $\Aut(G)$ for $G$ locally Roelcke precompact, and then proceed to  the  result that $\Out(G)= \Aut(G)/\Inn(G) $ is t.d.l.c.\ for each oligomorphic group~$G$~\cite[Theorem   3.10]{Nies.Paolini:24}.}

  \subsection{A Polish topology on $\Aut(G)$ with continuous action on $G$}
   \ \\
    A \emph{Polish space} is a separable topological space with a topology that is induced by a complete metric. A \emph{Polish group}  is a topological group based on a Polish space.   
    
  \begin{question} \label{qu: Pol} For which    Polish group $G$  can one topologise $\Aut(G)$ as a Polish group so that the   action   $\Aut(G) \times G \to G$ is continuous? \end{question}

  We develop a criterion for when the question has an affirmative answer. We will   apply it to the  locally Roelcke precompact non-Archimedean groups~\cite{Rosendal:21,Zielinski:21}. This class is large in the sense that is contains the disparate classes of  t.d.l.c.\ as well as   oligomorphic groups.  
  
 The setting is this:  Suppose that    there is a countable basis $\+ W(G)$ for the topology of $G$  such that  $\+ W(G)$ is invariant under the  natural action of $\Aut(G)$ on the open subsets of $G$. Then the action yields a monomorphism   $\Aut(G) \to \mathrm{Sym}( \+ W(G))$, where $\mathrm{Sym}(X)$ is the group of permutations on a set $X$. The group $\mathrm{Sym}(X)$ for countable $X$ is Polish via   the topology of pointwise convergence:  the pointwise stabilisers of finite subsets of $X$ form a neighbourhood basis of the identity, consisting of open subgroups.
  
  \begin{criterion} \label{fact:embed} Suppose the basis $
  \+ W(G)$ is as above, and the range of the induced  monomorphism   $\Aut(G) \to \mathrm{Sym}( \+ W(G))$    is closed. Then   Question~\ref{qu: Pol} has an affirmative answer  via the topology $\tau$ on $\Aut(G)$ that makes the embedding a homeomorphism.     Thus, one declares   as open the  subgroups of the form \begin{equation} \label{eqn: basis ident} \{\Phi  \in \Aut(G)\colon \forall {i= 1\ldots n} \  \Phi(A_i)= A_i\},\end{equation}  where $A_1, \ldots, A_n \in \+ W(G)$.
 \end{criterion}
\begin{proof} Note that   $\Aut(G)$  with this topology is non-Archimedean, and hence Polish.  An action of a Polish group on a Polish space is continuous iff it is separately continuous (see \cite{Gao:09}). So it suffices to show that for each $g\in G$, the map $\Aut(G) \to G$ given by $\Phi \mapsto \Phi(g)$ is continuous. Suppose then that $\Phi(g) \in B$ where $B\in \+ W(G)$. Let $A = \Phi^{-1}(B)$. Then $g\in A$. If $\Theta \in \Aut(G)$ is such that $\Theta(A)= \Phi(A)$, then $\Theta(g) \in B$. 
\end{proof}

 We next characterise the topology on  $\Aut(G)$ obtained, in the case that $G$ is totally  disconnected. Recall that a topological space $X$ is Baire if the intersection of countably many dense open sets is dense. Each Polish space is Baire.  
  
  \begin{prop} \label{prop:coarsest}  Suppose that    the elements of $\+ W(G)$ are all clopen, and   that  \cref{fact:embed} applies.  As above, let $\tau$ be   the  topology on $\Aut(G)$ 
  inherited from $\Sym(\+ W(G))$ through the monomorphism $\Aut(G) \to \mathrm{Sym}( \+ W(G))$. 
   \bi \item[(i)] Suppose  $\eta$ is a  Baire group topology on  $\Aut(G)$ making the action on $G$ continuous. Then $\tau \sub \eta$. 
    \item[(ii)]  $\tau$ is the unique Polish group topology on $\Aut(G)$ making the  action on $G$ continuous.   \ei\end{prop}

\begin{proof} (i)   Every  Baire measurable homomorphism from a Baire group $K$ to a separable group $L$  is continuous~\cite[Th.\ 9.9.10]{Kechris:95}.  (Here, to be  Baire measurable means that the pre-image of every open set  in $L$ has  meager  symmetric difference with some  open set in $K$.) We apply this to the identity homomorphism $(\Aut(G), \eta) \to(\Aut(G), \tau)$. If it is continuous then $\tau \sub \eta$. 
	
	Each Borel set for $\eta$ has meager symmetric difference with some open set in the sense of $\eta$ (i.e., it has the property of Baire for $\eta$).  We  verify  that   the $\tau$-open sets in the neighbourhood  basis (\ref{eqn: basis ident}) of the identity automorphism are $G_\delta$ for $\eta$: this  will imply that   each  $\tau$-open set is Borel for~$\eta$.   

To do so, for $A\in \+ W(G)$, we verify that  the  set 
\bc $\+ S=\{ \Phi\in \Aut(G) \colon \, \Phi(A) = A\} $  \ec is $G_\delta$ with respect to  $\eta$. Let $D$ be a countable dense subset of $A$. Then \bc $\+ S= \bigcap_{g \in D} \{\Phi \colon \ \Phi(g) \in A \lland \Phi^{-1}(g) \in A\}$.   \ec Thus $\+ S$  is $G_\delta$ for $\eta$ because $A$ is open and $\eta$ makes the action continuous.  Now each set in (\ref{eqn: basis ident}) is a finite intersection of such sets, and hence $G_\delta$ for $\eta$ as well.  

\n (ii) If $\eta$ is Polish, then it is Baire, so $\tau \sub \eta$. This implies   $\tau = \eta$ because no Polish group topology on a group can be properly contained in another (see, e.g., \cite[2.3.4]{Gao:09}). 
\end{proof}

\subsection{Polish topology on $Aut(G)$ for locally Roelcke precompact $G$}
 \n \emph{From now on we restrict ourselves to   the context of infinite closed subgroups of $\S$}; by $G,H$ we usually denote such groups.  For background on (locally) Roelcke precompact groups,  also outside the setting  of  non-Archimedean groups,  see Rosendal~\cite{Rosendal:21} and also   Zielinski~\cite{Zielinski:21}. 	
  \begin{definition} (i) A group $G$ is \emph{Roelcke precompact (R.p.)} if each open subgroup $U$ has only finitely many double cosets $UgU$, where $g\in G$.  In other words, the left action of $U$ on its own \emph{left} cosets has only finitely many orbits. 
 	
 	\n (ii)  $G$ is \emph{locally Roelcke precompact} if it has a R.p.\  open subgroup. \end{definition}

 Trivially, each open subgroup of a R.p.\ group is also R.p. Thus, if $G$ is locally R.p., its R.p.\ open subgroups form a neighbourhood basis of the identity. 
 
We obtain    an affirmative answer to   Question~\ref{qu: Pol} for locally Roelcke precompact groups
by establishing  a Polish topology on $\Aut(G) $ that makes its action on $G$ continuous. By~\cref{prop:coarsest} this topology is  unique.   In the following let $G$ be locally R.p/
\begin{definition} Given  a locally R.p.\ group $G$, let $\+ W(G)$ be the set of cosets of Roelcke precompact open subgroups, together with $\ES$. \end{definition}
Since $G$ is locally R.p.\ this is a basis of $G$. To show that  $\+ W(G)$ is countable, it suffices  to verify that it contains only countably many subgroups. Each such subgroup  $U$ is the union of finitely many double cosets of the point stabiliser of a finite subset of $\NN$, so there are only  countably many possibilities for $U$. 
	To show that  the range of the embedding $\Aut(G) \to \mathrm{Sym}( \+ W(G))$  is closed, we provide  a structure  (in the sense of model theory) on the basis, also  denoted     $\+ W(G)$ so that the range of the embedding equals $\Aut(\+W (G))$.   On $\+ W(G)$, take the intersection operation, together with the partial binary operation $A \cdot B$ that is defined when $A $ is a left coset and $B$ a right coset of the same subgroup~$U$, in which case  it is the usual product of subsets of~$G$. This was sketched already in~\cite[Section 3]{LogicBlog:22}.   $\+ W(G)$ seen as a structure  is called the \emph{meet groupoid} of $G$. 
	
	\begin{definition}[Full filters, \cite{Nies.Schlicht.etal:21}, Def.\ 2.4] \label{def:FF}
	A filter  $R $ on  the partial  $ (\+ W(G) - \{ \ES\}, \sub)$ is  called \emph{full}  if each open subgroup of $G$ has a  left and a right coset in $R$.  (Note that these cosets are necessarily unique.)
\end{definition}

\begin{lemma}  \label{lem:KechrisNies}
 There is a canonical bijection between the elements of a locally R.p.\ group $G$ and   the set of full filters on $  \+ W(G)$. It is given by \bc $g \mapsto R_g :=  \{ gU \colon \, U \le_o G\}$. \ec It inverse is given by $R \mapsto  g$ where $\{ g\}= \bigcap R$.  
\end{lemma}
\begin{proof}
The main point is to show that   $\bigcap R$ is non-empty for each full filter~$R$.  This  is proved similar to  \cite[Claims 3.6 and 3.7]{Kechris.Nies.etal:18}; also see the proof of \cite[Prop.\ 2.13]{Nies.Schlicht.etal:21}.
\end{proof}

\begin{lemma} \label{fa:when in R}
	Let $R$ be a full filter on $ \+ W(G)$. For each $A \in  \+ W(G) - \{ \ES\}$ we have 
	\bc $A\in R \LR \ex g \in A  [R =R_g] $. \ec
\end{lemma}
\begin{proof} Suppose  $A$ is a left coset of  the open subgroup $V$.
	For  right to left, note that  $A = gV \in R$. 
	
	For left to right, by Lemma~\ref{lem:KechrisNies} we have  $R = R_h =  \{hU \colon U \le_o G\}$  for some $h \in G$.  Then $hV \in R$, so $hV= A$, so $h \in A$.  
\end{proof}

\begin{prop} \label{prop: basis} Let  $G$ be locally R.p. Then  the hypothesis of  Criterion~\ref{fact:embed} holds: the range of the natural monomorphism $\Gamma\colon \Aut(G) \to \mathrm{Sym}( \+ W(G))$ is closed.
 \end{prop} 
 \begin{proof}  
We  show that the range of $\Gamma$ equals $\Aut(\+W (G))$, which is closed, being  the automorphism group of a structure.  	Given $\aaa \in \Aut(G)$, recall that   its image $\Gamma( \aaa)$ is the map  $A \mapsto \aaa(A)$.  Write $M$ for $\+ W(G)$ seen as a meet groupoid.   Clearly   $\Gamma( \aaa) \in \Aut(M)$ for each $\aaa \in \Aut(G)$.

 	We define the inverse $\Delta \colon \Aut(M) \to \Aut(G)$ by 
 	\bc $\{ \Delta(\Theta)(g)\}= \bigcap \{\Theta(g U)\colon  U \le_o G\}$.\ec
 	This is well defined because of Lemma~\ref{lem:KechrisNies}. 
 	We verify $\Delta(\Gamma( \aaa))= \aaa$ for any    $\aaa \in \Aut(G)$. Write $\hat \aaa = \Gamma(\aaa)$. Note that  \bc $\{\Delta(\hat \aaa)(g)\}=  \bigcap_{V \le_o G} \{\hat \aaa(g \hat \aaa^{-1}(V))\}= \bigcap_U \aaa(g)U=\{\aaa(g)\}$. \ec
 	We next  verify $\Gamma (\Delta(\eta))= \eta$ for any    $\eta \in \Aut(M)$.  Recall that we write $R_g$ for   the full filter  $\{gU \colon U \le_o G\}$ corresponding to $g$. For each $A \in M-\{\ES\}$, using Lemma~\ref{fa:when in R}, and  for convenienve  ignoring the difference between $x$ and~$\{ x\}$,
 	\begin{eqnarray*}
 		\Gamma (\Delta(\eta))(A)  & =&   \{ \bigcap \eta(R_g) \colon g\in A\} \\ 
 	&=& \{ \bigcap \eta(R) \colon \, R \text { full filter on } M \land A \in R\} \\
 	&=& \{ \bigcap S \colon \, S \text { full filter on } M \land \eta(A) \in S\} \\
 	&=&  \eta(A).
 	\end{eqnarray*} 
 \end{proof}

\subsection{Recovering $G$ from its meet groupoid $\+ W(G)$. }

As before suppose that $G$ is locally Roelcke precompact. Write $M= \+ W(G)$. Let

$\+ G(M)=\{ p \in \Sym(M) \colon \, p \text{ is a partial order isomorphism } \land $

 \hfill  $ p(A \cdot B) = p(A ) \cdot B \text{ whenever }   A \cdot B \text{ is defined } \}$.

\begin{remark} The elements of $\+G(M)$ are equivalent to full filters in the sense of Def.\ \ref{def:FF},  as follows: given $p\in  \+G(M)$ let $R = \{ pU \colon U \le_o G\}$. Given a full filter $R$, let $p(U) = A$ where $A$ is the left coset of $U$ in $R$; this determines the values on all cosets because $p(B) = p(U)\cdot B$ when $B$ is a right coset of $U$. One can see our treatment here as a simplified version of the treatment in \cite{Nies.Schlicht.etal:21}, which at the same time  applies to the much more general context of locally R.p.\ groups.  \end{remark}

 
\begin{lemma} $G \cong \+ G(M)$ via the map $L$ sending $g \in G$ to its left translation on~$M$: $L(g)(A)= gA$. Its inverse $K$   is given by sending $p \in \+ G(M)$ to $g\in G$ such that \bc $\{ g \} = \bigcap \{p(U)  \colon \, U \le_o  G\}$. \ec \end{lemma}
\begin{proof}
$L$ is a   group monomorphism because the action is faithful. $L$ is continuous because we only have check that for fixed $A$,  right coset of $U$, the set $\{ h \colon hA = A\}$ is open in $G$; this  equals $U$. 

  We show   that the inverse $K$   is well defined,  using Lemma~\ref{lem:KechrisNies}. Given $p\in \+ G(M)$ let $R = \{ p(U) \colon U \le_o G\}$. Then   $R$ is directed downward because $p$ is an order isomorphism. For the second hypothesis of Lemma~\ref{lem:KechrisNies}, let $U \le_o G$ be given. Clearly $p(U)$ is  a left coset of   $U$ because $p(U) = p(U) \cdot U$. Now  let $B = p^{-1}(U)$ which is also a left coset of $U$. So $V =B \cdot U \cdot B^{-1}$ is an open subgroup of $G$. Then  $p(V) \in R $ and $p(V)$ is a right coset of $U$, because $p(V) = p(B)\cdot U \cdot B^{-1} = U\cdot B^{-1}$.  
  
  It is routine to check that $L$ and $K$ are inverses. Firstly, we have $\{ K \circ L (g) \}= \bigcap_{U \le_o G} gU = \{g\}$. Secondly,  suppose $A$ is a right coset of $V$. We have  	$L \circ K (p)(A)=\bigcap_{U \le_o G} p(U) A \sub p(V)A = p(A)$. So they are equal,    being both  left  cosets of $V$.
\end{proof}

 \subsection{Some    Borel classes of  closed subgroups}
 Fig.\ \ref{fig:diagram} displays some Borel classes that are invariant under conjugation in $\S$. The arrows denote inclusion between classes.

 \begin{figure}[h] \label{fig:diagram}

\[ \xymatrix { && \text{locally R.p.} && \\
&\text{R.p.}\ar[ur] & & \text{locally compact}\ar[ul]&  \\
\text{oligomorphic}\ar[ur] && \text{compact}\ar[ul]\ar[ur] && \text{discrete}\ar[ul] 
 }
\]
\caption{Some Borel classes of infinite closed  subgroups of $\S$}
  \end{figure}
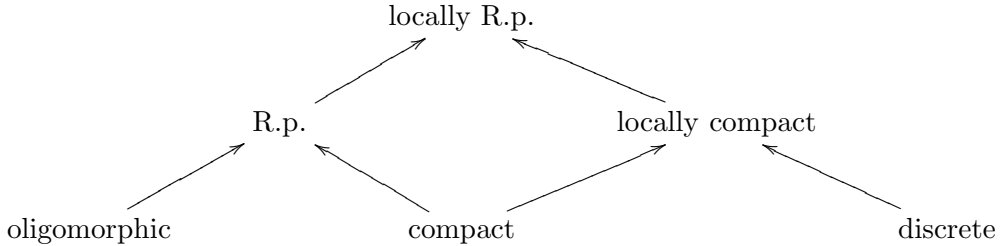
  First we will provide the formal definitions of  the classes in the diagram.

\begin{definition}[Cameron~\cite{Cameron:90}] A closed subgroup $G$ of $\S$ is \emph{oligomorphic} if for each $n$, the action of $G$ on $\NN^n$ has only finitely many orbits.  \end{definition} 
\begin{fact} [\cite{Tsankov:12}, Theorem 2.4]  A closed subgroup $G$ of $\S$ is oligomorphic if, and only if,  $G$ has only finitely many $1$-orbits and $G$ is  Roelcke precompact.  \end{fact} Note that being oligomorphic is a permutation group notion, while the other notions are about topological groups. A group $G$ (always closed subgroup of $\S$) is called \emph{quasi-oligomorphic}~\cite{Nies.Schlicht.etal:21} if it is homeomorphic to an oligomorphic group. The topological inverse limits of such groups are up to homeomorphism  the Roelcke precompact groups~\cite[Prop 2.2]{Tsankov:12}. This implies that the only R.p.\ abelian groups are the compact ones, using the fact that an oligomorphic group has finite centre. 

The automorphism group of a rooted infinitely branching tree is R.p.\ but not quasi-oligomorphic. It is   R.p.\ by~\cite[Th.\ 7 and discussion thereafter]{Zielinski:21}. 
Alternatively, one can write it as an inverse limit of the inverse system  $(G_n, p_n)$, where $G_n $ is the group of automorphisms of the ball of radius $n$ around the root, and $p_n \colon G_{n+1} \to G_n$ is the restriction map. Clearly $G_n$ is oligomorphic, in fact can be written as an iterated wreath product $(\ldots ((\S\wr \S) \ldots  \wr  \S) $, where $\wr$ denotes the permutation wreath product given the natural action of $\S$ on $\NN$.
So the automorphism group of  an infinitely branching tree is locally R.p., and not in any of the subclasses in Fig.\ \ref{fig:diagram}. 

It is instructive to  note that the  two classes in the middle level of Fig.\ \ref{fig:diagram}  intersect to  the class of compact groups. This was observed in~\cite[Section~2]{Tsankov:12}, and in fact holds without the hypothesis of being totally disconnected.
\begin{fact} If $G$ is t.d.l.c.\ and Roelcke precompact, then $G$ is compact. \end{fact}
\begin{proof} Let $V$ be a compact open subgroup of $G$ by van Dantzig's theorem. We consider the left action of $V$ on its   left cosets  in $G$. For each such coset $gV$, the stabiliser $\{x \in V \colon xgV= gV\}$ is open because  it equals  $V \cap \{ x \colon x^g \in V\}$. So it has finite index in $V$, hence  the  orbit of $gV$ under this action  is finite. Also, since $G$ is R.p.\  there are only finitely many orbits. Thus $V$ has finite index in~$G$.
\end{proof}

   \begin{remark}  In a  [quasi]-oligomorphic group, every open subgroup is R.p.  In a t.d.l.c.\ group $G$, the R.p.\ open subgroups coincide with the  compact open  ones. So the domain of the   meet groupoid  $\+W(G)$  consists of the compact open cosets. 
\end{remark}

  For a locally compact group $G$, the group $\Aut(G)$  carries the \emph{Braconnier topology} (see  \cite[Appendix I]{Caprace.Monod:11}). It is given by   neighbourhoods of the identity of the form $C(K,U)$, for any compact $K$ and open nbhd $U$ of $1$; here \bc $\aaa \in C(K,U)$ iff $\aaa(x)\in Ux \lland \aaa^{-1}(x)\in Ux$. \ec   If $G$ is discrete, then the Braconnier topology coincides with  the   topology of pointwise convergence on  $\Aut(G)$, because we can assume that $U $ is the trivial group, and of course each compact set is finite.   Clearly the   Braconnier topology   makes the action $\Aut(G) \times G \to G$ continuous. It is Polish~\cite[Appendix I]{Caprace.Monod:11}).  So by \cref{prop:coarsest} the Braconnier topology coincides with the one given by Criterion~\ref{fact:embed}.  A direct proof of this fact can be found in~\cite[Section 9]{LogicBlog:22}.

  \subsection{When is $\Inn(G)$ closed in $\Aut(G)$?}
The group of inner automorphisms $\Inn(G)$ is a normal subgroup of $\Aut(G) $. It  is of interest whether $\Inn(G)$ is closed in $\Aut(G)$, for in that case   the outer automorphism group $\Out(G)= \Aut(G)/ \Inn(G)$ is  a Polish group. 
  
 Wu~\cite{Wu:71}  gave an example of a discrete  group $L$ such that $\Inn(L)$ is not closed in $\Aut(L)$. A slight modification yields a nilpotent-2 group of exponent 3.   (Curiously, this group is finite automaton  presentable, and in fact similar to the groups   studied  recently in \cite{Nies.Stephan:24}.)
  
  \begin{example}[Similar to Example 4.5 in~\cite{Wu:71}] Let $G$ be generated by  elements of  $a_i,b_i, c$ order $3$ ($i \in \NN)$, where $c$ is central, with the relations  $b_ia_i b_i^{-1}= a_ic $, $[a_i,a_k] = [a_i,b_k]= [b_i,b_k]= 1$ for $i \neq k$.
  An automorphism $\Phi$ of $L$ is given by $\Phi(a_i)= a_i c$, $\Phi(b_i)= b_i$, and $\Phi(c) =c$. It  is not inner, but in the closure of $\Inn(G)$.
  \end{example}
  \begin{proof} To  check that $\Phi$ is indeed in $\Aut(G)$, note that  $c$ can be omitted from the list of generators. Given a   word $w$ in $a_i,b_k$  where each letter occurs with exponent $1$ or $2$, use that  $\Phi(w)= wc^{k \mod 3}$ where $k$ is the number of occurrences of   $a_i$'s. The inverse of $\Phi$ is given by $a_i \mapsto a_i c^{-1}$ and the rest as before. 
  
Write  $g_n = \prod_{i< n} b_i$. We have  $g_n a_i g_n^{-1}= a_ic$ for each $i< n$, and $g_n b_ig_n^{-1} = b_i$ for each $i$. Letting $\Phi_n$ be conjugation by $g_n$, we have $\lim_n \Phi_n = \Phi$ in $\Aut(G)$. 
For each $g\in G$, conjugation by $g$ fixes almost all the generators. So $\Phi$ is not inner. 
  \end{proof}
  
 Things look brighter on the left side of Figure~\ref{fig:diagram}. 
  \begin{thm}[Thm.\ 2.7 in \cite{Nies.Paolini:24}]  \label{th: Rp case} Given a  Roelcke precompact group closed subgroup $G$ of $\S$, the normal subgroup  $\Inn(G)$ is closed in $\Aut(G)$. \end{thm}
  \begin{proof}  We  follow the proof  for  oligomorphic groups~\cite[Theorem   3.3(i)]{Nies.Paolini:24}. The   open cosets of $G$ form a groupoid (a small category for which  all morphisms are invertible). For subgroups $U,V$ of $G$, by $\mathrm{Mor}(U,V)$ we denote the set of  right cosets of $U$   that are left cosets of $V$.

  \n {\bf Claim 1.} \emph{$\mathrm{Mor}(U,V)$ is finite, for any open subgroups $U,V$  of $G$.} 
  
  \n To see this, first note that $\mathrm{Mor}(U,U)$ is finite because each   coset in it is a double coset of $U$. Now suppose that there is a $B \in \mathrm{Mor}(U,V)$. There is a bijection $\mathrm{Mor}(U,U) \to \mathrm{Mor}(U,V) $ via $A \mapsto A \cdot B$. Thus $\mathrm{Mor}(U,V)$ is finite as required. This shows the claim.
  
  A   fact  in the  theory of Polish groups  states that each $G_\delta$ subgroup of a Polish group is closed (see \cite[Prop.\ 2.2.1]{Gao:09}); this relies  on the Baire category theorem. So  it suffices to show that $\Inn(G)$ is  a $G_\delta$ subset of  $\Aut(G)$.   Let $(A_n)_{n \in \omega}$ be a listing of the open cosets of $G$ without repetition. Given $\Phi \in \Aut(G)$,  we will define  a set $T_\Phi$ of strings  over some infinite alphabet which is closed under prefixes, and thus can be seen as  a  rooted tree. The alphabet consists of   pairs {of open cosets of $G$. To define  the $n$-th level we think of approximating some $g\in G$ such that $g A_i g^{-1} = \Phi(A_i)$ for $i< n$. At the $n$-level we   have  certain    pairs $(B_i,C_i)$ of approximations, where $i<n$:} 
 
 \[ T_\Phi= \{ \langle (B_i, C_i )\rangle_{i<n} \colon \Phi(A_i)= B_i \cdot A_i \cdot C_i^{-1}   \land \bigcap_{i <  n} (B_i \cap C_i) \neq \emptyset \}.\]
The domains of such open cosets $B_i$ and $C_i$ are determined by the condition that $\Phi(A_i)= B_i \cdot A_i \cdot C_i^{-1}$. By Claim 1, there are only finitely many possibilities for cosets with given  left and right domain. So the tree   $T_\Phi$ is   finitely branching.

  \n {\bf Claim 2.} \emph{$\Phi \in \Aut(G)$ is inner iff  $T_\Phi$ has an  infinite path.}    \newline
For the verification see Claim 2 in the proof of~\cite[Theorem   3.3(i)]{Nies.Paolini:24}, which goes through for   the general setting of R.p.\ groups. Assuming  the claim, we conclude the argument as follows. By K\"onig's Lemma, $T_\Phi$ has an infinite path iff each of its levels (i.e., strings of a   length $n$)  is nonempty. Whether the $n$-th level of $T_\Phi $ is nonempty only depends on the values   $\Phi(A_0), \ldots, \Phi(A_{n-1})$, so the set of such $\Phi$ is open in $\Aut(G)$. Thus the condition that each level is nonempty is $G_\delta$.  So  the claim   shows that the subgroup  $\Inn(G)$ is $G_\delta$ in $\Aut(G)$. 
  \end{proof}

  \begin{remark} If $G$ is R.p., then the topology that  $\Inn(G)$ inherits from $\Aut(G)$ is the   one given by quotient topology of $G$ by the centre. For, the centre $Z=Z(G)$   is closed. We have a Polish   topology on 
 $ G /Z$ by declaring $U Z /Z$ open iff $UZ$ is open in $G$. 
 The canonical  isomorphism $L \colon G/Z \to \Inn(G)$ is continuous. Since $\Inn(G)$ is Polish as a closed subgroup of $\Aut(G)$,  $L$ is a homeomorphism by a standard result in the theory of Polish groups (see e.g., \cite[2.3.4]{Gao:09}).  \end{remark} 
  
 \begin{remark}  Suppose that $G$ is merely locally R.p., and that   each R.p.\ open subgroup $U$  has finite index in its normaliser. Then $\mathrm{Mor}(U,U)$  is finite. So the proof of Th.\ \ref{th: Rp case} goes through. 
 \end{remark} 
 
  \begin{example}   $G=\Aut(\QQ, <)$ is oligomorphic.  The map $\pi(x) = -x$ is in the normaliser $N_G$ and not in $G$. A continuous    automorphism of $G$ not in $\Inn(G)$ is therefore given by $g \mapsto \pi^{-1} \circ  g \circ \pi$.  This is the only one up to $\Inn(G)$:  it is known that  $\Out(G)\cong N_G/G$ which has two elements \cite[Cor.\ 1.6]{Paolini:24}. \end{example}

  The following is~\cite[Theorem 3.1]{Nies.Paolini:24}.  \begin{thm} Let $G$ be oligomorphic. Then $N_G/G   $   is profinite,  where $N_G$ is the normaliser of $G$ in $\Sym(\omega)$.  \end{thm}
\begin{proof}  We  use a model-theoretic technique: an oligomorphic $G$ equals $\Aut(M)$ for the so-called  \emph{canonical structure} for $G$ (made out of the $n$-orbits for all $n \ge 1$). This $M$ is $\aleph_0$-categorical, which implies that the $G$-invariant relations coincide with the  first-order definable relations. Also, $N_G= \Aut(\+ E_M)$ where $\+ E_M$ is a reduct of $M$, its so-called \emph{orbital structure}, which consists of   the $n$-orbit equivalence relations for each $n$. It suffices to show  that $\mrm{Aut}(\mathcal{E}_M)/\mrm{Aut}(M)$ is a profinite group. 

 $M$ is interdefinable with $M_{\Aut(M)}$, so we may assume that   for every $0 < n < \omega$, $M$ has   predicates $P^n_1, ..., P^n_{k(n)}$ for each of its $n$-orbits, and no other relations or functions. For   $n \leq \omega$ let $M_n$ be the structure obtained after removing the predicates of arity greater than $n$, and let $\mathcal E_n:= \mathcal E_{M_n}$ be the corresponding orbital structure.
	
\smallskip \noindent	 For $k \leq m \leq \omega$ the identity map  $\mathcal E_{m} \to \mathcal E_k $  induces a map \bc $q_{m,k} \colon \Aut(\+ E_m )/ \Aut (M_m) \to \Aut(\+ E_k )/\Aut(M_k)$.  \ec Writing $p_n= q_{n+1,n}$ this yield an inverse system 
	$$(\mrm{Aut} (\+ E_n)/\Aut(M_n) , p_n)_{n \in \mathbb N}.$$ 
	The group $\mrm{Aut}(\mathcal{E}_n)/\mathrm{Aut}(M_n)$ is canonically isomorphic to a subgroup of $S_{k(n)}$  (the group of permutations of $\{1, \ldots, k(n)\}$): an isomorphism is induced by  mapping  $\pi \in \mrm{Aut}(\mathcal{E}_n)$ to the permutation $\alpha\in S_{k(n)}$ such that $\pi(P^n_i)= P^n_{\alpha(i)}$ for each $i \leq k(n)$. Let 
  $$  L: =  \varprojlim_n (\mrm{Aut} (\+ E_n)/\Aut(M_n) , p_n)$$
  be the inverse limit in the category of topological groups, where the finite groups $\mrm{Aut}(\mathcal{E}_n)\mathrm{Aut}(M_n)$ carry the discrete topology.  Every $f \in L $ can be concretely seen as a certain sequence of permutations $(\alpha_n)_{0< n < \omega}$ where $\alpha_n\in S_{k(n)}$. 
 \emph{We claim  that $R \cong L$ via the continuous homomorphism  $F\colon R\to L$ induced by the   maps $q_{\omega, n} : R \to \Aut(\+ E_n )/\Aut(M_n)$}. To verify this,  we   invoke the universal property of the inverse limit. To see that   $F$ is 1-1, suppose $\pi \in \Aut(\+ E_\omega) \setminus \Aut(M)$. Then there is $n$ and $i \leq k(n)$  such that $\pi(P^n_i)= P^n_j$ with $i \neq j$. So the image of $\pi \Aut(M)$ in  $\mathrm{Aut} (\+ E_n)/\Aut(M_n)$ is not equal to the identity. Thus $F(\pi \Aut(M))$ is not equal to the identity. 
	
\smallskip \noindent		To see that $F$ is onto,  given $f$ let $(\alpha_n)_{0< n < \omega}$ be as above, and  let $T_f$ be the theory in the language of $M$ with an additional function symbol $\pi$  with the following axioms: 
  
 \begin{enumerate}  \item  $Th(M)$;
\item $\pi$ is a permutation; 
  
\item  $\pi(P^n_i)= P^n_{\alpha_n(i)}$,   for each $n$ and each $i \leq k(n)$.\end{enumerate}
Note that $T_f$ is consistent  by the Compactness Theorem, using  the assumption that each $\alpha_n$ is induced by an automorphism $\pi_n$ of $\+ E_n$. Let $N$  be  any countable model   of $T_f$, and let $N'$ be the reduct of $N$ to the language of $M$. Since $M$ is $\aleph_0$-categorical, there is an isomorphism $h \colon M \to N'$. We have  $g=h^{-1} \circ \pi^N \circ h \in \Aut(\+ E_M) $ and $F( g \Aut(M))= f$. As $f\in L$  was arbitrary this shows that $F$ is onto. 

\smallskip \noindent The group isomorphism $F^{-1}\colon L\to R$ is clearly Borel, so it  is   continuous as a Borel (and hence Baire measurable) homomorphism between Polish groups; see, e.g., \cite[Th.\ 2.3.3]{Gao:09}.

\smallskip \noindent If $M$ has finite language then the inverse system above is eventually constant, and hence its  inverse limit is finite.
\end{proof}

\begin{remark}[needs checking] We note that the profinite group  $L$ can be ``read off" the theory of $M$. Assume for simplicity that all orbits are infinite.  From $\mrm{Th}(M)$ we can obtain the ternary  function $\theta$  on $\omega$      such that, for each positive $n$,   $r \leq k(n)$ and $i \leq n+1$,   the relation obtained by erasing the $i$-th component of the tuples in the $(n+1)$-orbit $P^{n+1}_r$ equals $P^n_{\theta(n,r,i)}$. Then   $L$  consists of the sequences of permutations  $(\alpha_n)_{0< n < \omega}$ that cohere with taking these projections in the sense  that, for each $n,r,i$ as above, we have \[\aaa_n(\theta(n,r,i))= \theta (n,\aaa_{n+1}(r),i).\]
For  let $N$ be the model such that $(P^n_i)^N= (P^n_{\aaa_n(i)})^M$ for each $n$ and $i \le k(n)$. Then $N \models \mrm{Th}(M)$ because the theory has quantifier elimination. Since the theory is $\aleph_0$-categorical, there is an isomorphism $\pi \colon M \to N$. Then $\pi \in \Aut (\+ E_M)$ and $\pi \Aut(M)$ is mapped to the sequence  $(\alpha_n)_{0< n < \omega}$. 
\end{remark}

By \cite{Evans.Hewitt:90} together with \cite{Nies.Paolini:24}, every separable profinite group can be realised  as $\Out(G)= N_G/G$: the structure constructed in~\cite{Evans.Hewitt:90} has ``no algebraicity'', which implies by \cite[Th.\ 4.7]{Nies.Paolini:24} that each automorphism of $G$ is  given as conjugation by a permutation  $\pi \in N_G$. 
 
The following is~\cite[Theorem 3.10]{Nies.Paolini:24}.

  \begin{thm} \label{th:compact} Let $G$ be oligomorphic.  Let $R \colon N_G \to \Aut(G)$ be the homomorphism defined by  $R(\pi)(g) = \pi \circ g \circ \pi^{-1}$. 
  
  \n 
 (i) The subgroup   $\mrm{range} (R)$ is open in $\Aut(G)$. 
 
 \n (ii) The group      $\Out(G)$ is t.d.l.c., having the compact  group  $\mrm{range} (R)/\Inn(G)$ as an open subgroup.   \end{thm}
   
   \begin{proof}
    \n {\bf  (i)}.  For $b,a \in \omega$, we will write $[b,a]$ for the coset $\{ g \in G \colon b=g(a)\}$.   Note  that for $a,b,c \in \omega$,  if $[c,a]\supseteq [b,a] \neq \emptyset$ then $b=c$.

\smallskip 
\noindent Let  $a_1, \ldots, a_n$ represent the $1$-orbits of $G$. It suffices to show that the open subgroup $\{ \Phi \in \Aut(G) \colon \, \Phi(G_{a_i})= G_{a_i} \text{ for each } i\}$ is contained in $\mrm{range} (R)$.  
 Suppose that   $\Phi $ is in this subgroup.  Since $\Phi$ fixes each subgroup $G_{a_i}$, for each coset $D=[b,a_i]$, $\Phi(D)$ is also a left coset of $G_{a_i}$. Hence, by~\cref{coset and pairs}, $\Phi(D)$  can be written in the form $[d,a_i]$. 
 Define a function $\pi_\Phi  \colon \omega \to \omega$  by 
 \[ \pi_\Phi(b)= d \text { if } \Phi([b,a_i]) = [d,a_i], \text{ where } [b,a_i ] \neq \emptyset. \]
 Clearly  $\pi_\Phi$ and $\pi_{\Phi^{-1}}$ are inverses. So $\pi_\Phi\in \Sym (\omega)$.    The following establishes (i).
 
 \smallskip
   \n {\bf Claim.} \emph{  $\pi_\Phi \in N_G$ and $R(\pi_\Phi) = \Phi$. } \newline 
    \smallskip
 Write $\pi=\pi_\Phi$. To verify the claim, we first show that   $\Phi([r,s]) = [\pi(r),\pi(s)]$  for each $r,s \in \omega$ such that $[r,s]\neq \emptyset$. 
Note that by hypothesis on $\Phi$ and since $G_{a_i} = [a_i,a_i]$, we have  $\pi(a_i)= a_i$ for each $i$. Also note that $\Phi([b,c])^{-1}= \Phi([c,b])$  for each $b,c \in \omega$. 
 Let now $i$ be such that $r$ and $s$ are in the $1$-orbit of $a_i$. We have $[r,s] = [r,a_i]\cdot [a_i,s]= [r,a_i]\cdot [s,a_i]^{-1}$. So 
 \[\Phi([r,s])= \Phi([r,a_i])\cdot \Phi([s,a_i])^{-1}= [\pi(r), a_i]\cdot [\pi(s, a_i]^{-1} = [\pi(r), \pi(s)]. \]

Next,  for each $h \in G$ we have $\{ h\}= \bigcap_{s\in \omega}  hG_s$, and $hG_s = [h(s), s]$.   Then 
 $ \{\Phi(g)\} = \bigcap_{s\in \omega} \Phi( [g(s), s])= \bigcap_{s\in \omega} [\pi(g(s)), \pi (s)] = \bigcap_{t\in \omega}  [\pi(g(\pi^{-1}(t))), t]= \{ g^\pi\}$,
 as required.  This shows the claim. 
  
\smallskip 
\noindent  {\bf  (ii)}.   The kernel of $R$ is the centraliser $C_G$ of $G$ in $\Sym(\omega)$, which is a normal subgroup of $N_G$.
This consists of the  permutations that are definable in the canonical structure for $G$, so $C_G$ is a finite.   $GC_G$ is a normal subgroup of $N_G$.  $GC_G$ is also closed in $\Sym(\omega)$ because  $G$ has finite index in it. Since $C_G$ is the kernel of $R$ and $R(G)= \Inn(G)$, we have $R^{-1}(\Inn(G)) = G C_G$. Thus   $\mrm{range}(R)/\Inn(G)$ is  topologically  isomorphic to $N_G/GC_G$. By~\ref{th:compact} $N_G/G$   is compact. So $N_G/GC_G$ is also compact as its topological quotient. 
   \end{proof}

   To summarise, we have  for each oligomorphic group $G$, 
   \[ \Inn(G) \cong G/Z(G)   \le_c N_G/ C_G \le_o Aut(G). \]
   Taking the quotient by $\Inn(G)$ on the right, we get a t.d.l.c. group with its  compact open subgroup predicted by van Dantzig's theorem: 
   
  \[ N_G/GC_G \le_o \Out(G).\]
  
  The work  with Philipp Schlicht   started during  his Auckland postdoc that took place  March-June 2024. It attempts to obtain the result Theorem~\ref{th:compact}(ii) in a   model-theoretic way: the elements of $\Out(G)$ can be seen as ``homotopy" equivalence classes of self-bi-interpretations of the canonical structure for $G$. Bi-interpretations may   improve our understanding of    $\Out(G)$, towards resolving for instance   whether it can be infinite and discrete. 
At present it's  known using~\cite{Evans.Hewitt:90} that $\Out(G)$ for oligomorphic $G$ can be any given separable profinite group. We don't have an example of an oligomorphic $G$ such that  $\Out(G)$ is not compact.

  \n {\emph Acknowledgement.} This section    is related to   a  talk Andr\'e Nies gave at a 2024 \href{https://austms.org.au/event/willis-fest-2024/}{conference} at the Tocal Agricultural College,  in honour of George Willis' birthday.

     \section{The Chabauty space of a t.d.l.c.\ group via its  meet groupoid}
In the following let $G$ be a t.d.l.c.\ group, and let $\+ W= \+ W(G)$ be its meet groupoid defined in~\cite[Section 4]{Melnikov.Nies:22}, which  in this setting is an algebraic structure  with domain the set of \emph{compact} open cosets of $G$, and otherwise following~\cref{def:meetgr}. Recall that $A,B,C,D$ always denote elements of $\+ W-\{\ES\}$. We  will show that the duality between t.d.l.c.\ groups and their meet groupoids holds for  the Chabauty space $\+ S(G)$  of $G$, in the sense that a homeomorphic copy of $\+ S(G)$ can be canonically recovered from~$\+ W$.

\subsection*{The Chabauty space}  We discuss the Chabauty space of a locally compact group $G$.  See e.g., Cornulier~\cite{Cornulier:11} for further background.

Using the notation of \cite[Section 2]{Cornulier:11}, a basic open set in $\+ S(G)$ has the form 
\begin{equation} \label{eqn: BO} \Omega(K; R_1, \ldots, R_n) = \{ U \le_c G \colon U \cap K = \ES \land  \forall i\le n \,  U \cap R_i  \neq \ES\},\end{equation}
where $K \sub G$ is compact, and the $R_i \sub G$ are open. 
 It is well-known that $\+ S(G) $ is compact; if $G$ is t.d.l.c.\ then $\+ S(G)$ is also  totally disconnected. So unless  $\+ S(G)$ has  isolated points, it  will be  homeomorphic to  Cantor space.

 \begin{remark} If $U_r \in \+ S(G)$ for $r\in \NN$, and $U_r \supseteq U_{r+1}$, then $\lim_r U_r= \bigcap_r U_r$. So for $\ZZ^n$, we get a countable Boolean space of rank $n$.
 Also, let $V_r \sub V_{r+1}$ for each $r\in \NN$ and let $V$  be the closure of $\bigcup_r V_r$. We have $\lim_r V_r= V$. To see this suppose that  $V\in \Omega(K; R_1, \ldots, R_n)$; we want to show that   $V_r\in \Omega(K; R_1, \ldots, R_n)$ for sufficiently large $r$. Clearly $V_r \cap K =  \ES$.  Since $R_i \cap V$ is nonempty, and  open in the subspace topology on $V$, we have $R_i \cap \bigcup_r V_r \neq \ES$ for each $i$. 
 \end{remark}
   
  A group  $P$ is called \emph{adic} if $P$ is a finite product of $\ZZ_n$'s and a finite group.
 Given  an  abelian locally compact group $G$, a closed subgroup $H$ is isolated in $\+ S(G)$ iff  
\bi \item[(t.d.l.c. case)] $H \cong A \times P$ where $A$ is f.g.\ torsion free, $P$ is adic, and $G/H$ is a finite product of Prufer groups (equivalently, it satisfies the descending chain condition on subgroups, i.e., Artinian), OR
 
\item[(elliptic case)] $H$ is adic  and $G/H \cong T \times D$, where $T $ is a finite power of the unit circle,  and $D$ is Artinian.  \ei
See    \cite[Thm.\ 1.7]{Cornulier:11}.
 
 For example,  $\ZZ_p$ is isolated in $\+ S(\QQ_p)$.    In fact, all   nontrivial closed subgroups of $\QQ_p$ are isolated, being either $\QQ_p$ or of the form  $p^r \ZZ_p$ for $r \in \ZZ$; $\{0\}$ is the only limit point of $\+ S(\QQ_p)$.

\subsection*{Ideals in a meet groupoid} $G$ will always denote a t.d.l.c.\ group, and $\+ W= \+W(G)$ its meet groupoid. 
Recall that the compact open cosets form a basis for the topology of $G$.  For $A, B_1, \ldots, B_n \in \+ W$, the relation $A \sub \bigcup_iB_i$ is  first-order definable in    $\+ W$, because  its negation is given by the formula 
  $\exists C \sub A  \bigwedge_i [C \cap B_i= \ES]$. So $\+ W$ determines a set of ideals in the following sense.

\begin{definition} We say that a set $\+ J \sub \+ W$ is an \emph{ideal} of $\+ W$ if  $\ES \in \+ W$ and \bc $(B_1, \ldots, B_n \in \+ J \lland A \sub \bigcup_iB_i)$ $\RA$  $A \in \+ J$. \ec \end{definition} 
\n Similar to Stone duality between Boolean spaces and Boolean algebras, open subsets $R$ of $G$ naturally correspond to ideals $\+ J$  of $\+ W$ via the maps

\bc $R \mapsto  \{ A \colon \, A \sub R\}$ and $\+ J \mapsto  \bigcup \+ J$. \ec
We also  write  $S_\+ J = G - \bigcup \+ J$, which is   a closed set in $G$.

\begin{definition} Let $\+ J$ be an ideal closed under inversion $A \mapsto A^{-1}$.

\n (i)  We say that   $\+ J$ is an \emph{open-subgroup ideal} if  
\bc  $A\cdot B$ defined  $\lland A \in \+ J  \lland  B \in \+ J  \RA A \cdot B \in \+ J$. \ec

\n (ii) We say that   $\+ J$ is a \emph{closed-subgroup ideal} if 
\bc $A\cdot B$ defined $\lland A\cdot B \in \+  J$ $\RA$ $A \in \+ J  \lor B \in \+ J$. \ec
  \end{definition} 
It is immediate that  $\+ J$ is open-subgroup ideal $\LR$  $R_\+ J$ is an (open)  subgroup of $G$. 
\begin{lemma} $\+ J$ is closed-subgroup ideal $\LR$ $S=S_\+ J$ is a (closed)   subgroup of $G$. \end{lemma}  \begin{proof} \lapf Immediate. 

\rapf Let $g,h \in G$. If $g^{-1} \not \in S$ then there is $A\in \+ J$ such that $g^{-1} \in A$. Since $\+ J$ is closed under inversion, this implies $g \not \in S$.  

If $g, h \in S $ but $gh \not \in S$ then pick $C \in \+ J$ such that $gh \in C$. By continuity there are $A, B\in \+ W$ such that $g\in A, h \in B$ and $AB \sub C$ (the   product of subsets in $G$). Let $A$ be left coset of $U$ and $B$ be right coset of $V$. Let $L = U \cap V$. Replacing $A$ by $\hat A = gL$ and $B$ by  $\hat B = Lh$, we have $\hat A \cdot \hat  B \sub C$ and hence  $\hat A \cdot \hat  B \in \+ J$. But $\hat A \cap S \neq \ES \neq \hat B \cap S$, so neither $\hat A$ nor $\hat B$ are in $\+ J$, contrary to the assumption that $\+ J$ is a closed-subgroup ideal.
  \end{proof}

Note that the set $\+ {CSI}(\+ W)$ of closed-subgroup   ideals   is closed in the product topology on $\+ P(W)$.

\begin{prop} The map $\Gamma \colon \+ {CSI}(\+ W) \to \+ S(G)$, given by \bc $\+ J \mapsto S_\+ J = G - \bigcup \+ J$,  \ec is a continuous bijection of compact spaces, and hence a homeomorphism. \end{prop}

\begin{proof} 
Fix a closed-subgroup ideal $\+ J$. To show that  $\Gamma$ is continuous at   $\+ J$, suppose that $S=\Gamma(\+ J) \in \Omega(K; R_1, \ldots, R_n) $,   a basic open set as defined in (\ref{eqn: BO}).  Since the compact open cosets form a basis of the topology of $G$, we   may assume that  each $R_i$ is a compact open coset.   Since $K \sub G-S$ which is open, there are compact open cosets $B_1, \ldots, B_m$ such that $K \sub \bigcup_k B_k \sub G-S$.   Let $\mathfrak L$ be  the basic open     set of           $\+ {CSI}(\+ W)    $ consisting of the subgroup ideals $\+ A$ such that $B_k \in \+ A$ for each $k\le m$, and $R_i \not \in \+ A$ for each $i \le n$. Clearly $\+ J \in \+ A$,  and                        $\+ H \in \mathfrak L$ implies                 $\Gamma(H) \in  \Omega(K; R_1, \ldots, R_n) $.
\end{proof}

Convergence on $\+ {CSI}(\+ W)$ is given by pointwise convergence: \bc $\+ J_r \to_r \+ J$ iff $\lim_r \+ J_r(A) = \+ J(A)$ for each $A \in \+ W$. \ec Therefore, denoting by $[\phi]$ the truth value of a statement $\phi$,   we have a convergence criterion in $\+ S(G)$:
\begin{cor} $\lim_r U_r = U \LR \forall A  \in \+ W \,  [U \cap A  = \ES] = \lim_r [U_r \cap A = \ES]  $.  \end{cor}

\subsection*{Computable setting}
Suppose that $G$ is given as a computable Baire presentation based on a c.l.c.\  tree $T$ as in~\cite[Def.~2.4]{Melnikov.Nies:22}. As discussed there,  after perhaps  modifying the presentation,  we may assume that $k=1$ in the definition, namely, only the root of $T$ can be infinitely branching. Let now $\+ W$ denote the corresponding Haar computable copy of  $\+ W(G)$ in the sense of  Melnikov and Nies~\cite[Def.~7.7]{Melnikov.Nies:22}. 

\begin{definition} (i) An open subset $R$ of $[T]$ is called computable if $\{\sss \in T \colon \, [\sss] \sub R\}$ is computable. 

\n (ii) A closed subset $S$ of $[T]$ is called computable if $[T]-S$ is computable. Equivalently,   the  subtree $\{\sss \in T \colon \, [\sss] \cap H \neq \ES\}$ corresponding to $S$ is computable. \end{definition} 

\begin{fact} For an ideal $\+ J \sub \+ W$, 

\bc $\+ J $ is computable $\LR$ the open set $R_{\+ J}\sub [T]$ is computable. \ec 
\end{fact}
\begin{proof} Suppose $T$ is clc via $k = 1$. 

\rapf Given $\sss \in T$, we may assume that $\sssl > 0$ (nonuniformly). So   $[\sss]_T$ is compact. Hence by~\cite[Lemma~2.6]{Melnikov.Nies:22} one can compute $A_1, \ldots, A_n\in \+ W $ such that  $\bigcup_i A_i = [\sss]$. Then $[\sss] \sub R$ iff $A_i \in \+ J$ for each   $i$.

\lapf By the  definition of  $\+ W$   from the computable Baire presentation based on the tree $T$, each $A  \in \+ W$ is given in the form $\+ K_u= \bigcup_{\eta\in u} [\eta]_T$ (as in~\cite[Def.~2.5]{Melnikov.Nies:22}) where $u$ encodes a finite set of nonempty strings on $T$. Then $A \in \+ J$ iff $[\eta]\sub R$ for each $\eta$ in $u$. 
\end{proof}

For the following cf.\ \cite[Fact~9.1]{Melnikov.Nies:22} which deals with computable closed subgroups. 

\begin{cor} In the context of the above, an open-subgroup ideal [closed-subgroup ideal]   $\+ J$ is computable    iff   $R_\+ J$ [resp., $S_\+ J$] is computable. \end{cor}

%
%
 
  \part{Metric spaces and computability theory}

 \section{Mayordomo and Nies:  \\ Fractal dimensions and profinite groups} 
By $X^*$ we denote the set of tuples over a set $X$, seen as a rooted tree where the successor relation is given by appending one element. Let    $T \sub \NN^* $ be a finitely branching  subtree (i.e., set of strings closed under initial segments)   such that each node has at least two successors. We view the path space $[T]$ as an ultrametric space. We   will use the point-to-set principle in metric spaces of Lutz et al.\ \cite{Mayordomo.etal:23} to show   that    lower box dimension and Hausdorff    dimension are equal for certain  ``regular" subsets of $[T]$, namely path spaces $[S]$ where $S $ is a subtree of $T$ such that  the number of successors is equal across each  level. 
 
We apply this to profinite groups by giving a new proof of Barnea and Shalev \cite[Thm.\ 2.4]{Barnea.Shalev:97}. 
Let the group $G = \varprojlim \sN n \, (C_n, p_n)$ be the inverse limit of an inverse system of finite groups, where $p_n \colon C_{n+1} \to C_n$ are onto homomorphisms.   $G$ is  a compact group with a compatible  ultrametric, where the distance between distinct elements $f, g$ of $G$ is $1/|C_n|$ for the least $n$ such that the projections of $f$ and $g$  into $C_n$ are different. 
The lower box dimension of a closed subgroup $U$ of $G$  is easily seen to be   
 $\liminf_i (\log u_i  /\log|C_i|)$, where $u_i$ is the size of the projection of $U$ into $C_i$. 
  The  result \cite[Thm.\ 2.4]{Barnea.Shalev:97} of central importance  for their paper  says that  the Hausdorff dimension of a closed subgroup  $U$    of $G$ equals its lower box dimension.  They rely on Abercrombie~\cite[Prop.\ 2.6]{Abercrombie:94}  for     the nontrivial part that the  box dimension is no larger than the Hausdorff dimension. Abercrombie proved this via means specific to profinite groups, including Haar measure. We show here that it is purely geometric fact that doesn't need the group structure, but is merely exploiting the regularity of a description of the subgroup  as a subtree of the tree given by the inverse system.   
 
 \subsection{Preliminaries} 
\n We begin by  describing constructive dimension  in computable metric spaces, and the point-to-set principle for both Hausdorff and packing dimension  in that context. Here we mostly follow Lutz et al.\ \cite{Mayordomo.etal:23}. We   restrict ourselves   to  the standard gauge family $\theta_s(\aaa) = \aaa^s$, which has a precision family in their sense by   \cite[Observation 2.1]{Mayordomo.etal:23}.

 Suppose we are given a computable metric space $M$.  Thus,  there is dense set $D\sub M$, called the set of special points, and   there is  an onto  function $r \colon \{0,1\}^* \to D$ such that the distance  between special points $r(x), r(y)$ is uniformly computable in $x,y$ as a real number. (One says that   $D, r$ determine a \emph{computable structure} on $M$.)  For $w\in D$ we write $C^{M,r}(w)$  for $\min \{ C(u) \colon r(u) = w\}$.  
 
 \begin{definition}[\cite{Mayordomo.etal:23}, (3.1) and (3.2)] \label{df:ccdim} For  $x\in M$ and any $\aaa > 0$,  
 one defines the \emph{Kolmogorov complexity of a point $x$ at precision   $\alpha$} by  \begin{equation} \label{eqn: define Calpha}  C_\alpha(x) =
\min \{C^{M,r}(w) \colon  w \in D \lland d(x,w)\le \alpha\}.\end{equation}
 The \emph{algorithmic  dimension} of $x \in M$ is 
 \begin{equation} \label{eqn:constr dim}  \textit{dim}(x) :  = \liminf_{\alpha \to 0^+} C_\alpha(x)/ \log (1/\alpha).\end{equation}
  The \emph{strong algorithmic  dimension} of $x \in M$ is 
 \begin{equation} \label{eqn:constr dim}  \textit{Dim}(x) :  = \limsup_{\alpha \to 0^+} C_\alpha(x)/ \log (1/\alpha).\end{equation} \end{definition} 
 Let $\dim_P([X])$ denote the packing dimension of $X$. It is defined from ``inside" via disjoint sets of balls with centre in $X$. See e.g.\ \cite[Section 2]{Mayordomo:08} for the detailed definition.
The \emph{point-to-set principles} (PTSP)   \cite[Thms 4.1,4.2]{Mayordomo.etal:23}  determine the Hausdorff and packing dimension of a \emph{set} $X \sub M$ in terms of the relativised constructive dimensions of the\emph{ points} in it:
 \begin{eqnarray} \label{eqn:PTSP}  \dim_H (X) & = &  \min_A \sup_{x \in X} \textit{dim}^A(x),\\ \dim_P (X) & = &  \min_A \sup_{x \in X} \textit{Dim}^A(x) \label{eqn:PTSP2}  \end{eqnarray}

\begin{definition}  Fix some computable injection   $\pi \colon   \NN^* \to \{0,1\}^*$, say \bc $(n_0, \ldots, n_{k-1}) \mapsto 0^{n_0 } 1 \ldots 0^{n_{k-1}} 1$.   \ec    The plain descriptive (Kolmogorov) complexity $C(y) $ of $ y \in \NN^*$    is defined to be $C (\pi(y))$;  similarly we define $K(y)$, $C(y\mid n)$, etc. The choice of $\pi$ will only affect these values up to a   constant.
It will be  clear from context whether $C(y)$ is taken to be evaluated for $y \in \NN^*$ or for $y \in \{0,1\}^*$. There are  two interpretations of $C(y)$ for    a binary string $y$; however they agree up to  a constant. \end{definition}

\begin{convention} All logarithms will be taken with respect to base $2$. \end{convention}
 \subsection{The space of paths on  a rooted tree.}  Recall that  $T \sub \NN^*$ is a finitely branching  tree such that each node has at least two successors. By $T_n$ we denote the set of strings in $T$ of length $n$. The path spaces $[T]$ of  such  trees are precisely the compact subsets of Baire space. We view them as metric spaces via  the ultrametric 
 \[ d(f,g)= \max \{ |T_n|^{-1} \colon \, f(n) \neq g(n)\}\]
 where $\max \ES = 0$ by convention. 

  Suppose $[T]$ is computable. We     describe an appropriate computable structure on the metric space $[T]$.  Let $r_0 \colon \{ 0,1\}^* \to T$ be the   computable bijection that is an isomorphism between  the length lexicographical orderings on $\{ 0,1\}^*$ and $T$. Given a binary string~$u$, let  $r(u)$ be  a path on $[T]$ extending the string  $ r_0(u)$; we can choose it  uniformly computable in $u$ since $T$ is computable.
  
   Recall   the definition of $C_\aaa(x)$ in   (\ref{eqn: define Calpha}). We will relate it to the initial segment  complexity of paths.
 \begin{prop} \label{cl: T1} Suppose that     $T\sub \NN^*$ is a computable tree such that each node has at least two successors. Let $r \colon \{ 0,1\}^*\to [T]$ define a computable structure as above. Let $f \in [T]$.  For each~$\aaa$ such that  $|T_{n-1}|^{-1} >\alpha \ge |T_n|^{-1}$, we have  \bc $ C(f\uhr n) - 2 \log n \lep  C(f\uhr n \mid n) \lep C_{\alpha}(f)\lep   C(f\uhr  n)$. \ec \end{prop}
\begin{proof}  The leftmost inequality is standard. For the middle inequality,  
   suppose    $w \in D$ and  $d (w, f) \le \aaa$. Given $u \in  \{ 0,1\}^* $ such that $r(u)= w$ and $n\in \NN$, we can compute $w \uhr n$,  which  has to equal $f\uhr n$ since $|T_{n-1}|^{-1} >\alpha$.  Thus $ C(f\uhr n \mid n) \lep C(w)$.

 For the rightmost inequality,   given a string $y = f\uhr n \in T$ we   compute the  string $u\in \{0,1\}^*$ such that $r_0(u)= y$. Since $d(f, r(u)) \le \alpha$, this shows $C_\aaa(f) \lep C(u)=^+ C(y)$. \end{proof}

 \subsection{Constructive dimension on path space $[T]$} We express the values  $\dim(f)$ and $\textit{Dim}(f)$ (Def.\ \ref{df:ccdim}) of a path $f \in [T]$ in terms of its initial segment complexity. This is of course well known at least for the case of a binary tree $T$, where it goes back to \cite{Mayordomo:02}.
\begin{prop}  \label{cl: T2} Suppose each node on $T$ has at least two successors. For any  $f\in [T]$   we have  \[\textit{dim}(f)= \liminf_n \frac{C(f \uhr n) }{ \log |T_n|} = \liminf_n \frac{K(f \uhr n) }{ \log |T_n|}\] 
 \[\textit{Dim}(f)= \limsup_n \frac{C(f \uhr n) }{ \log |T_n|} = \limsup_n \frac{K(f \uhr n) }{ \log |T_n|}\]

\end{prop}
\begin{proof} We only verify the first assertion; the argument for the second assertion is almost identical. By hypothesis    we have $\log|T_n| \ge n$, and so $\lim_n \log n/\log |T_n|=0$.  For the left  equality,  using  Prop.\ \ref{cl: T1}  we have   
\bc $\dim(f)=  \liminf_{\alpha \to 0^+} C_\alpha(f)/ \log (1/\alpha)= \liminf_{n \to \infty} C(f\uhr n)/ \log |T_n|$.  \ec
For the right  equality, note that for each $n$ and each $y \in T_n$, we have     $C(y) \lep K(y) \lep C(y) + K(n)\lep C(y) + 2 \log n$. 
\end{proof}

\subsection{ML-randomness and Levin-Schnorr Theorem on $[S]$.}
For a tree $S$ without leaves,  the path space $[S]$ carries the  uniform   measure $\mu_S$ given by     $\mu_S[S]=1$;  if $\mu_S([\sss]\cap [S] $ is defined, then  $\mu_S([\sss\ape i]\cap [S])= \mu_S([\sss]\cap [S])/k_\sss$ where $k_\sss$ is the number of successors of $\sss$ on $S$. If $S$ and the function  $\sss \mapsto k_\sss$ are computable, then $\mu_S$ is computable.  For such $S$, one can generalise the  Levin-Schnorr Theorem (see, e.g.,  \cite[3.2.9]{Nies:book}) to $[S]$.   For $X \sub S$ let $\ROpcl {S} X =  \{ f \in S \colon \exists n f \uhr n \in X\}$ be the open subset of $[S]$ described by $X$.
 Let \[ R_{S,b} = \ROpcl {S} { \{ \sss \in S \colon K(\sss) \le -  \log \mu_S[\sss] - b\}}.\] 
 \begin{lemma} \label{lem:SL}  $(R_{S,b})\sN b$ is a universal ML-test on $[S]$.
 \end{lemma}
 
 \begin{proof} Clearly the open sets $R_{S,b}$ are  $\SI 1$ uniformly in $b$. Given $b$, let  $D$ be the set of prefix minimal strings $\sss \in S$ such that $K(\sss) \le - \log \mu([\sss])-b$. Then \[ 1 \ge \sum_{\sss \in D} \tp{-K(\sss) } \ge \sum_{\sss \in D}\tp{\log \mu([\sss]) + b}.\]
 This implies that $\mu_S \ROpcl S D \le \tp {-b}$.  
 
 Universality is shown as in the proof for Cantor space $\cantor$; see \cite[3.2.9]{Nies:book}.  To follow that proof, one first needs to write down a machine existence theorem for strings  in $\NN^*$. 
 \end{proof} 
 
 While we are mainly interested in the case that  $[S]$ is uncountable,  the argument above also works  for countable $[S]$. 
 Note that in the countable case only the empty set is null, so every path is ML-random on $[S]$.
 \medskip
 

\subsection{Box dimensions of bounded closed subsets.} The lower and upper box   dimension of a bounded  set $X$ in a metric space $M$ are
\begin{eqnarray} \label{eqn:lower box} \ul \dim_B(X)&=&
\liminf_{\alpha \to 0^+} \frac{\log N_\alpha(X)}{ \log (1/\alpha)} \\ 
\ol \dim_B(X)&=&
\limsup_{\alpha \to 0^+} \frac{\log N_\alpha(X)}{ \log (1/\alpha)} \nonumber
\end{eqnarray}
where $N_\alpha(X)$ is the least cardinality  of a covering of $X$ with sets of diameter at most $\alpha$. 
 Now let  $M =[T]$ for a tree $T$ as above, and  write $t_n = |T_n|$. Each closed subset of $[T]$  is given as $[S]$ for a unique subtree $S$ of $T$.  
 If $t_{n-1}^{-1} > \alpha \ge t_n^{-1} $, then  we may assume that the sets in   a covering as above  are of the form $[\sss]$ for $\sssl =n$, because  these are the maximal  sets of diameter at most $\aaa$. So,  we have $N_\alpha(X)= |S_n|$. Also, if  $t_{n-1}^{-1} > \alpha \ge t_n^{-1} $ then  $-\log \alpha \le \log t_n$. So  

\begin{eqnarray} \label{eqn:dimB}  \ul \dim_B([S])&=& \liminf_{n \to \infty} \frac{ \log |S_n|}{\log |T_n|}\\ 
\ol \dim_B([S])&=& \limsup_{n \to \infty} \frac{ \log |S_n|}{\log |T_n|}. \nonumber \end{eqnarray}

\begin{definition} \label{df:ubr}  We say that a tree $S \sub \NN^*$ is \emph{uniformly branching} if for each $n$, each $\sss \in S_n$ has the same number of successors on $S$. \end{definition}
Thus, $\mu_S([\sss]) = |S_n|^{-1}$ for each $\sss \in S_n$. 

\begin{example} Let $T = \{0,1 \}^*$. Let $S$ be the uniformly branching subtree of $T$ such that for  each even $k$, for each $n \in [4^k,4^{k+1})$, strings at level $n$ have one successor, and for  each odd $k$, for each $n \in [4^k,4^{k+1})$, strings at level $n$ have two successors. Then $\ul \dim_B([S]) \le 1/4$ and $\ol \dim_B([S])\ge 3/4$.
\end{example}
\begin{prop} \label{cl:T3}  (i) Let $S$ be a computable subtree of $T$, which is a computable tree as above.  For each    $f \in [S]$,
\begin{eqnarray*} \textit{dim}(f) &\le &  \ul \dim _B(S), \text{ and}  \\
\textit{Dim}(f) &\le&   \ol \dim _B(S).    \end{eqnarray*}

\n (ii) Suppose $S$ is uniformly branching. Then equalities hold in the case  that $f$  is \ML\ random in $[S]$ with respect to~$\mu_S$.  \end{prop} 
\begin{proof} We use Proposition~\ref{cl: T2} for both parts. 

\n (i)  Let $\eta_S \colon S \to \{0.1\}^*$ be    the  (computable) isomorphism between the  length-lexicographical orderings on $S$ and on $\{0,1\}^*$.  
Recall that   to define $C(y)$ and $K(y)$ for $y \in \NN^*$, we fixed a  computable injection $\pi \colon \NN^* \to \{0,1\}^*$. Also recall that $\log$ is in base $2$.
We claim that for   $y \in S_n$,  \bc $K(y)  \le^+ \log |S_{\le n} |+ 2 \log \log |S_{\le n} |$.   \ec  To verify  this, 
let $x = \eta_S(y)$.  We have     $|x| \le \lceil \log |S_{\le n}|\rceil$.  Since $x$ is binary,   $K(x) \le^+ |x| + 2 \log |x|$. Since $K(y) = K(\pi (\eta_S^{-1}(x)))\lep K(x)$ the claim follows. 

To verify  the  inequalities, let $y = f\uhr n$ in the above.  Note that \bc $\log |S_{\le n} | \le \log (n \cdot  |S_n|) \lep \log n + \log |S_n|$. \ec So $K(y) \le \log n  + \log |S_{n} |+ \log ( \log n +  \log |S_n |)$. Now use ,  that     $\log n + \log ( \log n +  \log |S_n |)  = o(\log |T_n|) $.

\n (ii) If $f$ is ML random then $f$ passes the universal test $(R_{S,b}) \sN b$ in  Lemma~\ref{lem:SL}). Thus  $\log |S_n| =^+ -\log \mu_S[f\uhr n]  \lep K(f\uhr n)$.  
The equalities now follow from   Proposition~\ref{cl: T2}.  
\end{proof}

\subsubsection*{Lower box  dimension of $[S]$ equals Hausdorff dimension  for  a   uniformly branching subtree $S$.}  We  compare Hausdorff and lower box dimension, as well as packing and upper box dimension. 
It is known that $\dim_H(X) \le \ul \dim_B(X)$ and $\dim_P(X) \le \ol \dim_B(X)$  for any bounded subset $X$ of a metric space $M$. However, in our setting where $X=[S]\sub [T]=M$, and $S$ uniformly branching,  we won't need this, as the equalities come right out of our setting. 

\begin{thm} \label{prop: dims equal} Let $S\sub T \sub \NN^*$ be trees, and suppose  that  each node on $T$ has at least two successors. Suppose  that $S$ is uniformly branching.  Then  \begin{eqnarray*} \dim_H[S]  & =&  \ul \dim_B[S], \text{  and}   \\
\dim_P[S] &= &\ol \dim_B[S].\end{eqnarray*}
\end{thm}

\begin{proof} 
Suppose   that $E$ is an oracle such that $S \le_T E$, $T \le_T E$, and $E$ can compute  from~$n$ the   number  of successors for $S$ at   level~$n$. Note that  the uniform measure $
\mu_S$ on $[S]$ is $E$-computable.

We  show $\dim_H[S] =  \ul \dim_B[S]$.  Recall     the point-to-set principle for Hausdorff dimension  (\ref{eqn:PTSP}):  \[\min_A \sup_{f \in [S]} \textit{dim}^A(f)= \dim_H[S].\] 

So it suffices to show that $\min_A \sup_{f \in [S]} \textit{dim}^A(f)= \ul \dim_B([S])$. We have the inequality ``$\le$"  by Prop.\ \ref{cl:T3}(i) relative to $E$. For the converse inequality, note that 
in taking the minimum, we may assume that the oracle $A$ computes $E$, because replacing $A$ by $A \oplus E$  does not increase  $\textit{dim}^A(f)$. For each such $A$, if we choose $f $ to be  ML$^A$-random in $[S]$, by Prop.~\ref{cl:T3}(ii) we have  $\textit{dim}^A(f)=  \ul \dim_B(S)$.

The other equality is similar using  the point-to-set principle  for packing dimension (\ref{eqn:PTSP2}) instead
\end{proof}

%
%
%
In the ultrametric setting there isn't much known in fractal geometry about  the various notions of fractal  dimension. Books such as  Falconer's focus entirely on the case of subsets of $\RR^n$. We ask:
\begin{question} Is there a classical-style proof of Theorem~\ref{prop: dims equal}, just based on the definitions of $\dim_H$, $\dim_P$ and $\ul \dim_B$, $\ol \dim_B$?\end{question}

\subsection{Dimension of closed subgroups  in profinite groups}
We are given an inverse system of finite groups $(C_n)$ with onto   homomorphisms $p_n \colon C_{n+1} \to C_n$ that are not bijections. We may assume $C_0$ is the trivial group. The elements of the  group $G= \varprojlim_n (C_n, p_n)$ can be viewed as the paths on a  tree $T$, where the $n$-th level is  $T_n= C_n$, and $\sss\in T_n$ has as successors all the $\eta \in T_{n+1}$ such that $p_n(\eta)= \sigma$. For each $n$,  each $\sigma \in T_n$  has   $|C_{n+1}|/|C_n|$  successors.  So  $T$ is uniformly branching. With some encoding of the elements of each $C_n$ by the first $|C_n|$ natural numbers, we can view $T$ as a subtree of $\NN^*$.

Suppose now that $U$ is a closed subgroup of $G$. Let $U_n$  be the projection of $U$ into $C_n$. Then $U = \varprojlim (U_n,q_n)$  where $q_n= p_n \uhr {U_{n+1}}$ is  an onto homomorphism $U_{n+1} \to U_n$  for each $n$. Letting $S_n = \{ f\uhr n \colon f \in U\}$, $S$ is a subtree of $T$, the elements of $U_n$ can be identified with $S_n$,   and    $U$ can be  isometrically identified with $[S]$.   As above, $S$ is uniformly branching. The following 
 is now immediate from Theorem~\ref{prop: dims equal}.
\begin{thm}[\cite{Barnea.Shalev:97}, Thm.\ 2.4 for the first eqation]  In the notation above,  \[ \dim_H(U) = \ul \dim_B(U) = \liminf_n (\log |U_n| /\log|C_n|).\]

 \[ \dim_P(U) = \ol \dim_B(U) = \limsup_n (\log |U_n| /\log|C_n|).\]\end{thm}

For $p$-adic analytic groups $G$, by \cite[Thm 1.1]{Barnea.Shalev:97},  $\dim_H(U)=\dim(U) /\dim(G)$, where $\dim(X)$ is the dimension of a closed set $X$ in the $p$-adic manifold underlying $G$. At the end of Section 3 they assert that the proof  also works for the upper box dimension in place of $\dim_H$, so the same expression holds. Thus all the fractal dimensions agree for closed subgroups of $G$; in particular,  $\dim_H$ and $\dim_P$ agree.

 \part{Ergodic theory and computability}

   \section{Bhojraj: the Ergodic condition on cylinders and Borel sets}

Let $A$ be a finite alphabet, let $A^n$ be the set of strings of length $n$ from $A$, let $A^* = \bigcup_n A^n$ and let $A^\infty$ be the set of infinite sequences from $A$. Let $T$ be the shift operator on $A^\infty$ which deletes the first element of a sequence, let $\mathcal{B}$ be the Borel sigma algebra of $A^\infty$ and let $\mu$ be a measure on $(A^\infty, \mathcal{B})$. The set of cylinder sets $P= \{[y] :  y \in A^* \}$  generates $\mathcal{B}$ (i.e., $\sigma(P)=\mathcal{B}$). Consider the two statements:

\begin{equation}
\label{1}
    \forall X,Y \in \mathcal{B},\lim_{N \longrightarrow \infty} N^{-1} \sum_{k\leq N-1} \mu(X\cap T^{-k}(Y))= \mu(X)\mu(Y).
\end{equation}

\begin{equation}
\label{2}
    \forall x,y \in A^* ,\lim_{N \longrightarrow \infty} N^{-1} \sum_{k\leq N-1} \mu([x]\cap T^{-k}([y]))= \mu([x])\mu([y]).
\end{equation}
For any $X,Y \in \mathcal{B}$, let $E(X,Y)$ denote the statement that \[\lim_{N \longrightarrow \infty} N^{-1} \sum_{k\leq N-1} \mu(X\cap T^{-k}(Y))= \mu(X)\mu(Y)\] (E stands for `ergodic'). With this notation, \ref{1} says that  $\forall X,Y \in \mathcal{B}, E(X,Y)$ and \ref{2} says that $\forall x,y \in A^*, E([x], [y])$, that is, \ref{1} restricted to the generating set of  $\+ B$.   

\begin{thm}
    \ref{1} is equivalent to \ref{2}.
\end{thm}

\begin{proof}
    It suffices to show that \ref{2} implies \ref{1}. Fix a $y \in A^*$ and let \[G_{y}:= \{Z \in \mathcal{B}: E([y], Z)\}.\]
      Fix a $B\in \mathcal{B}$ and let \[F_B := \{Z\in \mathcal{B}: E(Z,B)\}.\]
      The plan is to first prove that $\mathcal{B}\subseteq G_{y}$ for all $[y]\in P$ by showing that $G_y$ is a Dynkin system and by using \ref{2} to note that $P\subseteq G_y$ for all $[y]\in P$. This shows that $P\subseteq F_B$ for all $B\in \mathcal{B}$. We then show that $F_B$ is a Dynkin system for all $B\in \mathcal{B}$ which implies that $\mathcal{B}\subseteq F_B$ for all  $B\in \mathcal{B}$.
      \begin{lemma}
          $G_y$ is a Dynkin system for all $y \in A^*$.
      \end{lemma}
    \begin{proof} 
        We note that $A^\infty \in G_y$ trivially and then prove that it is closed under taking countable increasing unions and under taking the difference between a set and its subset. Let $Z_1 \subseteq Z_2 \subseteq \cdots$ be an increasing sequence of elements from $G_y$ and let $Z:= \bigcup_i Z_i$. Define 
        \[f_m(n):= n^{-1}\sum_{k\leq n-1} \mu([y]\cap T^{-k}(Z_m)),\]
        and 
        \[F(n):= n^{-1}\sum_{k\leq n-1} \mu([y]\cap T^{-k}(Z)).\]
        The plan is to apply the ``interchange of limits" 
        Theorem \ref{app}:   if a sequence of   functions $(f_m)_m$, where $f_m \colon \NN \to \RR$,  converges uniformly to a function $F$ from below,  and if each $f_m$ converges to a $L_m$ as $n \rightarrow \infty$, then the limit of the $L_m$'s is $\lim_n F(n)$. I.e.,  $\lim_m (\lim_n f_m(n))=\lim_n (\lim_m f_m(n))$.   Note that $f_m\leq F$. We first show that $(f_m)_m$ converges uniformly to $F$: Fix an $\varepsilon >0$. As $Z_m$ tends to $Z$, fix a $M_\varepsilon$ such that $\mu(Z-Z_m)\leq \varepsilon$, and hence that $\mu(Z)-\mu(Z_m) \leq \varepsilon$, for all $m>M_\varepsilon$. As $T$ is $\mu-$preserving, we have that $\mu(T^{-k}(Z))-\mu(T^{-k}(Z_m))$, and hence that $\mu(T^{-k}(Z)-T^{-k}(Z_m))\leq \varepsilon$, for all $k$ and all $m>M_\varepsilon$. Finally,
        $\mu([y] \cap (T^{-k}(Z)-T^{-k}(Z_m)))\leq \varepsilon$ and hence that $\mu([y] \cap T^{-k}(Z))-\mu([y] \cap T^{-k}(Z_m))\leq \varepsilon$ for all $k$ and all $m>M_\varepsilon$. As this holds for all $k$,
\[ \forall m> M_\varepsilon \forall n, n^{-1}\sum_{k\leq n-1} \mu([y]\cap T^{-k}(Z_m))\geq n^{-1}\sum_{k\leq n-1} \mu([y]\cap T^{-k}(Z))-\varepsilon.\]
I.e.,
\[ \forall m> M_\varepsilon \forall n, f_m(n)\geq F(n) -\varepsilon.\]
As $f_m\leq F$,
\[ \forall m> M_\varepsilon \forall n, |F(n) - f_m(n)|\leq \varepsilon,\]
showing the uniform convergence. Note that $T^{-k}$ being $\mu$-preserving for all $k$ was used in the above. $Z_m \in G_y$ for each $m$ implies that $L_m:=\lim_n f_m(n) = \mu([y])\mu(Z_m)$ for each $m$. Also, $L=\lim_m L_m$ exists and is $\mu([y])\mu(Z)$ as $\mu$ is a measure. By Theorem \ref{app}, $\lim_n F(n)=L$ showing that $Z\in G_y$ and thus that $G_y$ is closed under countable increasing unions. Now, suppose $C \subseteq D$ and that both $C$ and $D$ are in $G_y$. Then, for any $k$, $T^{-k}(D-C) = T^{-k}(D)- T^{-k}(C)$ and $T^{-k}(C) \subseteq T^{-k}(D)$. So,
         \[ \lim_{N}N^{-1} \sum_{k\leq N-1} \mu([y]\cap T^{-k}(D-C)) \] \[ =  \lim_{N}N^{-1}\sum_{k\leq N-1} \mu([y]\cap T^{-k}(D)) - \lim_{N}N^{-1}\sum_{k\leq N-1} \mu([y]\cap T^{-k}(C))\]\[=\mu([y])\mu(D)- \mu([y])\mu(C) = \mu([y])(\mu(D) - \mu(C))= \mu([y])\mu(D-C). \]
         This shows that $G_y$ is closed under taking the difference between a set and its subset, thus completing the proof showing $G_y$ to be a Dynkin system.
    \end{proof}
    Condition \ref{2} implies that $P\subseteq G_y$. As $P$ is clearly a $\pi$-system, the $\pi-\lambda$ theorem implies that $\mathcal{B}=\sigma(P)\subseteq G_y$. As $y$ was arbitrary, we have that $\forall [y] \in P, \mathcal{B} \subseteq G_y$ and hence that
    \[\forall [y] \in P, \forall B\in \mathcal{B}, E([y],B),\]
    or, equivalently,
    \[ \forall B\in \mathcal{B}, \forall [y] \in P, E([y],B).\]
  The above then gives that $P\subseteq F_B$ for all $B\in \mathcal{B}$. If we show that $F_B$ is a Dynkin system for all $B\in \mathcal{B}$, then the $\pi-\lambda$ theorem will again imply that $\sigma(P)=\mathcal{B}\subseteq F_B$ for all $B\in \mathcal{B}$. This will show our main theorem. To this end, we show that:
  \begin{lemma}
      $F_B$ is a Dynkin system for any $B\in \mathcal{B}$.  
  \end{lemma}
 \begin{proof} The proof is similar to that of the previous lemma except that $T$'s being $\mu$-preserving, which was crucially used in the latter, is not needed here.
        We note that $A^\infty \in G_y$ trivially and then prove that it is closed under taking countable increasing unions and under taking the difference between a set and its subset. Let $Z_1 \subseteq Z_2 \subseteq \cdots$ be an increasing sequence of elements from $F_B$ and let $Z:= \bigcup_i Z_i$. Define real valued functions, $(f_m)_m$ and $F$ on $\omega$ by:
        \[f_m(n):= n^{-1}\sum_{k\leq n-1} \mu(Z_m\cap T^{-k}(B)),\]
        and 
        \[F(n):= n^{-1}\sum_{k\leq n-1} \mu(Z\cap T^{-k}(B)).\]
        Note that $f_m\leq F$. Now show that the $f_m$'s uniformly converge to $F$. Fix an $\varepsilon >0$. As $Z_m$ tends to $Z$, fix a $M_\varepsilon$ such that $\mu(Z-Z_m)\leq \varepsilon$ for all $m>M_\varepsilon$. So, $\mu((Z-Z_m)\cap T^{-k}(B))\leq \varepsilon$, and hence that $\mu(Z \cap T^{-k}(B))-\mu(Z_m \cap T^{-k}(B))\leq \varepsilon$, for all $k$ and all $m>M_\varepsilon$. As this holds for all $k$,
\[ \forall m> M_\varepsilon \forall n, n^{-1}\sum_{k\leq n-1} \mu(Z_m\cap T^{-k}(B))\geq n^{-1}\sum_{k\leq n-1} \mu(Z\cap T^{-k}(B))-\varepsilon.\]
So, noting that $f_m\leq F$,
\[ \forall m> M_\varepsilon \forall n, |f_m(n) - F(n)| \leq \varepsilon.\]
So, $f_m$ uniformly converges to $F$. $\lim_n f_m(n)=L_m=\mu(Z_m)\mu(B)$ exists for all $m$ as $Z_m \in F_B$ for all $m$. $L=\lim_m L_m$ exists as $\mu$ is a measure. Applying Theorem \ref{app} shows that $Z \in F_B$ similarly to the previous lemma. $F_B$ is thus closed under countable increasing unions. Now, suppose $C \subseteq D$ and that both $C$ and $D$ are in $F_B$. So, $\mu((D-C)\cap T^{-k}(B))= \mu(D\cap T^{-k}(B))-\mu(C\cap T^{-k}(B)).$  Using this and that $C,D \in F_B$, 

\smallskip

$\lim_{N}N^{-1} \sum_{k\leq N-1} \mu((D-C)\cap T^{-k}(B)) $
         \begin{eqnarray*}  \\ & = &  \lim_{N}N^{-1}\sum_{k\leq N-1} \mu(D\cap T^{-k}(B)) - \lim_{N}N^{-1}\sum_{k\leq N-1} \mu(C\cap T^{-k}(B))\\ 
         &=&\mu(B)\mu(D)- \mu(B)\mu(C)  = \mu(B)\mu(D-C). \end{eqnarray*}
         This shows that $F_B$ is closed under taking the difference between a set and its subset, thus completing the proof showing $F_B$ to be a Dynkin system.
    \end{proof}
    
    The main theorem is proved.
\end{proof}

\subsection{Appendix}
\begin{thm}
\label{app}
    Let $(f_{m})_m$ be a sequence of real valued functions on $\omega$ such that
    \begin{enumerate}
    
        \item There is a real valued $F$ on $\omega$ such that $f_m \leq F$ and the sequence $(f_m)_m$ converges uniformly to $F$.
        \item $L_m := \lim_n f_m(n)$ exists for each $m$ and $L:=\lim_m L_m$ exists.
    \end{enumerate}
    
     Then, $\lim_n F(n)$ exists and equals $L$.
\end{thm}

\begin{proof}
Fix an $\varepsilon >0$ and a $M_\varepsilon$ such that $||f_m - F||_\infty \leq \varepsilon$ for $m>M_\varepsilon$. So,
\[ \forall m> M_\varepsilon \forall n, F(n)-\varepsilon\leq f_m(n).\]
Take the limsup as $n$ tends to $\infty$ (for each fixed $m$) on both sides.

\[ \forall m> M_\varepsilon \limsup_n F(n)-\varepsilon\leq L_m.\]

As this holds for all $\varepsilon$,
\[L=\lim_m L_m \geq \limsup_n F(n).\]
We now lower bound $\liminf_n F(n)$ by $L$ to finish the proof. As
$f_m (n) \leq F(n), $ holds
for any $m$ and $n$, we have that for all $m$,
\[L_m = \lim_n f_m (n) \leq \liminf_n F(n).\]
Taking the limit over $m$ on the left gives that $L\leq \liminf_n F(n)$.
\end{proof}

\part{Set theory}

 \section{Lam and Nies: Whitehead's problem and the category of sets}
  SET is the category of sets with maps in between them.  Ka Ho Lam, a former student of Nies, asked in 2019 whether there are  properties of this category that are not decided by ZFC. Here we sketch how to obtain  such a property,  which is first order in the language of SET. We want the sentence to be ``natural" from the point of view of categories; in particular, it should be  invariant under equivalence of categories. This notion is due to MacLane, and will be  recalled below.   Speaking semantically, our result implies   that  there are two countable models $\+ M, \+ N$ of ZFC such that their respective categories SET$^\+ M$ and SET$^\+ N$   are not equivalent.

  An abelian group $A$  is called a \emph{Whitehead group} if   each extension of $A$ by $\ZZ$ splits. This means that for each exact sequence $0 \to \ZZ \to E \to A \to 0$, there is a morphism $A \to E$ inverting $E \to A$, or briefly $\Ext(A,\ZZ)= 0$. Clearly each free group is a Whitehead group. 
 The Whitehead problem asked whether the converse holds. By a result of Stein~\cite{Stein:51}, each countable Whitehead group is free, so  a counterexample wouild  necessarily be  uncountable. 
Shelah~\cite{Shelah:74,Shelah:75},  proved that this statement  is independent of ZFC. If $V=L$ then the conjecture holds. If $MA \lland \lnot CH$ holds then there is a Whitehead group that is not free.  (Here MA is Martin's axiom, and CH is the continuum hypothesis.)
Three years later Shelah~\cite{Shelah:77}   showed  that the problem can have a negative answer even in the presence of CH.

  To  obtain a sentence in the language of the category of sets which ZFC cannot decide, we will proceed  as follows. We first phrase Whiteheads problem by a first-order sentence  about the category of abelian groups. Then   we first-order  interpret this  category  in the category of sets. The sentence will be invariant under equivalence of categories.  
  
  Another way to obtain such a sentence is  to express the generalized continuum hypothesis GCH in the category of sets.  We need a subobject classifier $\Omega$, namely a 2-element set (which can be characterized up to isomorphism by a certain pullback property, as is well known in topos theory).  We  want to talk about monomorphisms  $\aaa\colon U \to \mathrm{Mor}(X,\Omega)$, which express  subsets of $\+ P(X)$.  We need to replace these by maps $\overline \aaa \colon U \times X \to \Omega$ and      express that $\aaa$ is a  monomorphism   via a  first order diagrammatic property of $\ol \aaa$.
  Also see ``Sheaves in geometry and logic by  Moerdijk and MacLane, pp.\ 277 ff.  
  
  The present solution makes the result reliant on Shelah's work~\cite{Shelah:77} but not the fact that  CH is independent of ZFC.

  \subsection{Background on categories, and first-order logic}
 
  We view categories as generalised structures where the domain can be a class (consisting of the morphisms); they only have  a ternary relation $R$ for composition of morphisms. $Rfgh$ means that $fg$ is defined and equals $h$; henceforth, we also  write this as $fg= h$. A morphism $f $ is an identity morphism if $\forall g    [\forall  h (R fgh \to h=g) \wedge \forall h (R gfh \to h=g )]$. Objects are identified with the corresponding  identity morphisms, so we don't need a special sort for them. In practice we use the usual letters $A,B,C, \ldots$ for them, and   write $1_A$ if we want to view $A$ as morphism,  e.g., in writing things like $ 1_A f =f$. Given a category, there are first-order definable  functions $\sss$ and $\tau$    that assign  to a morphism its source and target objects, respectively. We have 
  \bc $\sss(f)= h$ if $h \text{ is identity }  \land hf=f$ \ec 
  
  \bc  $\tau(f)= \ell$ if $\ell \text{ is identity } \land f\ell=f$. \ec We write \bc $\mor f A B $ for $\sss(f) = A \land \tau(f)= B$. \ec We sweep under the rug some  logical complications that ensue because   the domain of the ``structure" usually is a proper class. We merely  note that saying that a subclass, or a  function,  is first-order definable  via a  formula   means that ZFC proves the equivalence of the original definition  of the function (which was given in the usual    language of set theory) with what the defining formula (which is in the language of the category) expresses.
  
   Two categories $\+  C, \+ D$ are \emph{equivalent} (in the sense of MacLane) if there is a functor $T\colon \+ C \to \+ D$ which is 
  
 \bi \item  full and faithful, i.e. for any two objects $A, B$ of $\+ C$, the map \bc $\mathrm{Hom}(A,B) \to \mathrm{Hom} (TA,TB)$ \ec   induced by $T$ is surjective and injective;

\item essentially surjective (or dense), i.e.   each object $B $ in $\+ D$ is isomorphic to an object of the form $TA$, for some object $A$ in $\+ C $. \ei
  
Freyd~\cite{Freyd:76} (also see   \href{http://angg.twu.net/Freyd76.html}{this site})     characterised the first-order properties of categories  that are invariant under equivalence.
 These are the ones that  are   diagrammatic  (essentially: can be  defined  by a finite sentence in a language of diagrams and using  the   quantifiers $\forall, \exists$).  For instance, requiring the existence of a pushout of a diagram with three objects is diagrammatic; saying that there are two distinct objects that are isomorphic is not.
  
   \subsection{A statement about SET that is  undecidable in ZF, and invariant under equivalence of categories}
    Note that we write composition of maps in SET in the natural order, $(fg)(x)= g(f(x)$. In ZF (without requiring choice)  it is much easier to obtain an invariant  sentence in the language of SET that cannot be decided. 
We say that AC holds in a \emph{category}  if each epimorphism has a left inverse. That is, if $\mor f A  B$ is an epimorphism  then there is $\mor g B A$ such that $gf=1_B$.  It is well known that in ZF, the axiom of choice for SET is equivalent to the usual axiom of choice. The only extra trouble is due to the fact that we want the property to be invariant under equivalence. The next two lemmas verify this directly, without recurring to Freyd's work.

\begin{claim} Let $T \colon \+ C \to \+ D$ be an equivalence of categories. Suppose that  $\mor f A  B$ is  a morphism  in $\+ C$. Then  

  $f$ is an epimorphism in $\+ C$  $\LR$  

\hfill $T(f) $ is an epimorphism in~$\+ D$.   \end{claim}
\begin{proof} \rapf Suppose $\mor {\wt g_i} {T(B)} {D}$ in $\+ D$ for $i=0,1$, and $T(f)\wt g_0= T(f)\wt g_1$. Since $T$ is dense, there is an object  $C \in \+ C$ and an isomorphism $\mor p  {D} {T(C)}$.  Since $T$ is full, there are morphisms $\mor {g_i} B C$ in $\+ C$ such that $T(g_i)= \wt g_i p$. Since $fg_0= fg_1$ we have $ g_0= g_1$, and hence $\wt g_0= \wt g_1$. 

\lapf This implication   holds for any faithful functor. 

 \end{proof}
\begin{fact} Let $\+C$ and $\+ D$ be equivalent categories. Then   

\n  AC holds in $\+ C$   $\LR$ 
AC holds in $\+ D$.   \end{fact}
\begin{proof} Since $T$ has an inverse that is also an equivalence of categories, we only need to show $\LA$. So suppose $\mor f A B$ is an epimorphism  in $\+ C$. Then in $\+ D$ there is  $\mor {\wt g} {T(B)} {T(A)}$ such that $\wt g T(f)= 1_{T(B)}$. Since $T$ is full, $\wt g = T(g)$ for some $\mor g BA$. Then  $gf=1_B$ since $T$ is faithful. \end{proof}

  \subsection{The category $Ab$} 
 
 Recall that an abelian group $A$ is free if and only if  it is projective: for any epimorphism $\mor g UV$ and any morphism $\mor h AV$ there is $\mor f AU$ such that $fg = h$. The direction left to right holds for any algebraic structure: if $A$ is freely generated by a set $X$, let $f(x)$ for $x \in X$ be a preimage of $h(x)$ in $U$, and extend. For the   direction right to left, one uses that every subgroup of a free abelian group is itself free.
 
 The advantage of saying  ``projective" instead of ``free"  is that being projective is first-order definable in a  diagrammatic way.
 
%

\begin{prop} The class of Whitehead  groups  is first-order definable  in a diagrammatic way   in the  category     $Ab$ (over ZFC). \end{prop} 
\begin{proof}
 
1.   Given $\mor f AB$ one  can  express   that $f=0$: this holds exactly when $f g_0= fg_1$ for each $B$ and each pair of morphisms $\mor {g_i} A B$. 
 
2.   $\ZZ$ (up to isomorphism)  is the only  projective object not isomorphic to a proper direct product. 

3. A sequence $ \ZZ\stackrel f \to A  \stackrel g \to Q$ is exact exactly if  the monomorphism $\mor f \ZZ A $ is maximal such that $fg =0$; 

\n that is, if  there is a proper monomorphism $\mor h \ZZ  M$ and monomorphism $\mor {\wt f} M A$ such that $f = h \wt f$, then $\wt f g \neq 0$.  \end{proof}

So there is  a first-order  sentence $\theta$   in the language of $Ab$ expressing  that some  Whitehead group   is not projective. 

%
\begin{prop} The   statement $\theta$ is   preserved under equivalences of categories  in the sense of MacLane. \end{prop} 
This is because the definitions are   diagrammatic in the sense of Freyd as discussed above.
\begin{remark} \label{rem:ffgghh} We   conclude that the first-order theory of $Ab$ is not determined by ZFC, and (for Platonists) that there are two countable models of ZFC such that their categories $Ab$ are not equivalent. \end{remark}

 \subsection{The category $SET$}
First-order interpretations  of one struxture in another are fundamental in model theory; see the model theory books by Hodges (Ch.\ 5), or Tent and Ziegler. It is fairly straightforward to extend this notion to ``structures" with domain a proper class. We say that  an interpretation is   diagrammatic  if the involved first-order formulas are diagrammatic.  

\begin{prop}  There is a diagrammatic interpretation of the category $Ab$ in the category $SET$. \end{prop} 

\begin{proof}
%
An object of $Ab$  is represented by an 8-tuple of morphisms  of SET \bc  $(A, S, p_0, p_1,  a,\ell, m, s)$ \ec   where  $A, S$ are objects (identity morphisms), $a \colon S \to A$ represents  the binary  group operation, $\ell \colon A\to A$ is the inverse operation, $m \colon 0 \to A$, and the $p_i \colon S \to A$ are the projections. One can say  in a first-order diagrammatic way     that $S, p_0,p_1$ has the universal property describing the product $A \times A$.  One  requires that  the morphism \bc $s \colon S \to S$ \ec    switches the ``components" of the direct product: $s p_i= p_{1-i}$ for $i=0,1$.   One can now  formulate the abelian group axioms via diagrams, similar to the case of  Hopf algebras in \href{https://www.cs.auckland.ac.nz/~nies/Students/Dominik_Roje_Honours_Dissertation.pdf}{Dominik Roje's Honours thesis,} pages 19-21. For instance,  commutativity is expressed by $s a = a$.  

A morphism $g\colon A\to B$ is then just a set morphism that satisfies some ``commuting diagram property" expressing that it preserves the group operations.  	 Also see Boersuk, Handbook of categorical algebra.
\end{proof}

\begin{prop} Suppose that there is a diagrammatic    interpretation $\Gamma$ of a category  $\+ C_0$ in $\+D_0$. Let  $\+ D_0 $ be    equivalent to $\+ D_1$ and let   $\+ C_1$  be the category obtained as $\Gamma(\+ D_1)$. Then $\+ C_0$ is   equivalent to $\+ C_1$. \end{prop}

Thus, the internal categories $SET$  relative to   the   pair of countable models of set theory obtained in Remark~\ref{rem:ffgghh} are not   equivalent.

 \part{AI and maths}
 
 \section{Copilot is trying to solve the Basel problem} 
 Here's a LaTeX code snippet for a proof of the result that the sum of the reciprocals of the squares of the positive natural numbers equals $\pi^2/6$. This is known as the Basel problem, and one of the classic proofs involves using the Fourier series.

We want to show that 
\[
\sum_{n=1}^{\infty} \frac{1}{n^2} = \frac{\pi^2}{6}.
\]

%
%
%
%
%
%
%

Consider the Fourier series of the function \( f(x) = x^2 \) on the interval \([- \pi, \pi]\). The Fourier series expansion is given by:
\[
f(x) = \frac{\pi^2}{3} + 4 \sum_{n=1}^{\infty} \frac{(-1)^n}{n^2} \cos(nx).
\]

Evaluating this at \( x = \pi \), we get:
\[
\pi^2 = \frac{\pi^2}{3} + 4 \sum_{n=1}^{\infty} \frac{(-1)^n}{n^2} \cos(n\pi).
\]

Since \( \cos(n\pi) = (-1)^n \), the series becomes:
\[
\pi^2 = \frac{\pi^2}{3} + 4 \sum_{n=1}^{\infty} \frac{(-1)^n}{n^2} (-1)^n = \frac{\pi^2}{3} + 4 \sum_{n=1}^{\infty} \frac{1}{n^2}.
\]


Subtracting \(\frac{\pi^2}{3}\) from both sides, we get:
\[
\pi^2 - \frac{\pi^2}{3} = 4 \sum_{n=1}^{\infty} \frac{1}{n^2}.
\]

Simplifying the left-hand side:
\[
\frac{2\pi^2}{3} = 4 \sum_{n=1}^{\infty} \frac{1}{n^2}.
\]

Dividing both sides by 4, we obtain:
\[
\sum_{n=1}^{\infty} \frac{1}{n^2} = \frac{\pi^2}{6}.
\]

%
%

%

%
%

\begin{thebibliography}{10}

\bibitem{Abercrombie:94}
A.~Abercrombie.
\newblock Subgroups and subrings of profinite rings.
\newblock In {\em Mathematical {P}roceedings of the {C}ambridge {P}hilosophical
  {S}ociety}, volume 116, pages 209--222. Cambridge {U}niversity {P}ress, 1994.

\bibitem{Ahlbrandt.Ziegler:86}
G.~Ahlbrandt and M.~Ziegler.
\newblock Quasi finitely axiomatizable totally categorical theories.
\newblock {\em Annals of Pure and Applied Logic}, 30(1):63--82, 1986.

\bibitem{Barnea.Shalev:97}
Y.~Barnea and A.~Shalev.
\newblock Hausdorff dimension, pro-?? groups, and {K}ac-{M}oody algebras.
\newblock {\em Transactions of the American Mathematical Society},
  349(12):5073--5091, 1997.

\bibitem{Cameron:90}
P.~Cameron.
\newblock {\em Oligomorphic permutation groups}, volume 152.
\newblock Cambridge University Press, 1990.

\bibitem{Caprace.Monod:11}
P.-E. Caprace and N.~Monod.
\newblock Decomposing locally compact groups into simple pieces.
\newblock In {\em Mathematical Proceedings of the Cambridge Philosophical
  Society}, volume 150, pages 97--128. Cambridge University Press, 2011.

\bibitem{Cornulier:11}
Y.~Cornulier.
\newblock On the {C}habauty space of locally compact abelian groups.
\newblock {\em Algebraic \& Geometric Topology}, 11(4):2007--2035, 2011.

\bibitem{LogicBlog:18}
A.~Nies (editor).
\newblock Logic {B}log 2018.
\newblock Available at \url{http://arxiv.org/abs/1902.08725}, 2018.

\bibitem{LogicBlog:20}
A.~Nies (editor).
\newblock Logic {B}log 2020.
\newblock Available at \url{https://arxiv.org/pdf/2101.09508.pdf}, 2020.

\bibitem{LogicBlog:22}
A.~Nies (editor).
\newblock Logic {B}log 2022.
\newblock Available at \url{https://arxiv.org/pdf/2302.11853.pdf}, 2022.

\bibitem{Evans.Hewitt:90}
D.~Evans and P.~Hewitt.
\newblock Counterexamples to a conjecture on relative categoricity.
\newblock {\em Annals of Pure and Applied Logic}, 46(2):201--209, 1990.

\bibitem{Freyd:76}
P.~Freyd.
\newblock Properties invariant within equivalence types of categories.
\newblock In {\em Algebra, Topology, and Category Theory}, pages 55--61.
  Elsevier, 1976.

\bibitem{Gao:09}
Su~Gao.
\newblock {\em Invariant descriptive set theory}, volume 293 of {\em Pure and
  Applied Mathematics (Boca Raton)}.
\newblock CRC Press, Boca Raton, FL, 2009.

\bibitem{Hodges:93}
W.~Hodges.
\newblock {\em Model Theory}.
\newblock Encyclopedia of Mathematics. Cambridge University Press, Cambridge,
  1993.

\bibitem{Kechris:95}
A.~S. Kechris.
\newblock {\em Classical descriptive set theory}, volume 156.
\newblock Springer-Verlag New York, 1995.

\bibitem{Kechris.Nies.etal:18}
A.~S. Kechris, A.~Nies, and K.~Tent.
\newblock The complexity of topological group isomorphism.
\newblock {\em The Journal of Symbolic Logic}, 83(3):1190--1203, 2018.

\bibitem{Mayordomo.etal:23}
J.~Lutz, N.~Lutz, and E.~Mayordomo.
\newblock Extending the reach of the point-to-set principle.
\newblock {\em Information and Computation}, 294:105078, 2023.

\bibitem{Mayordomo:02}
E.~Mayordomo.
\newblock A {K}olmogorov complexity characterization of constructive
  {H}ausdorff dimension.
\newblock {\em Inform. Process. Lett.}, 84(1):1--3, 2002.

\bibitem{Mayordomo:08}
E.~Mayordomo.
\newblock Effective fractal dimension in algorithmic information theory.
\newblock In {\em New Computational Paradigms: Changing Conceptions of What is
  Computable}, pages 259--285. Springer, 2008.

\bibitem{Melnikov.Nies:22}
A.~Melnikov and A.~Nies.
\newblock Computably totally disconnected locally compact groups (full
  version).
\newblock Available at \url{arxiv.org/pdf/2204.09878.pdf}., 2022.

\bibitem{Nies:book}
A.~Nies.
\newblock {\em Computability and {R}andomness}, volume~51 of {\em Oxford Logic
  Guides}.
\newblock Oxford University Press, Oxford, 2009.
\newblock 444 pages. Paperback version 2011.

\bibitem{Nies.Schlicht.etal:21}
A.~Nies, P.~Schlicht, and K.~Tent.
\newblock Coarse groups, and the isomorphism problem for oligomorphic groups.
\newblock {\em Journal of Mathematical Logic}, page 2150029, 2021.

\bibitem{Nies.Stephan:24}
A.~Nies and F.~Stephan.
\newblock Word automatic groups of nilpotency class 2.
\newblock {\em Information {P}rocessing {L}etters}, 183:106426, 2024.

\bibitem{Paolini:24}
G.~Paolini.
\newblock The isomorphism problem for oligomorphic groups with weak elimination
  of imaginaries.
\newblock {\em Bulletin of the London Mathematical Society}, 2024.

\bibitem{Nies.Paolini:24}
G.~Paolini and A.~Nies.
\newblock Oligomorphic groups, their automorphism groups, and the complexity of
  their isomorphism.
\newblock {\em arXiv preprint arXiv:2410.02248}, 2024.

\bibitem{Rosendal:21}
C.~Rosendal.
\newblock {\em Coarse Geometry of Topological Groups}.
\newblock Cambridge Tracts in Mathematics. Cambridge University Press, 2021.

\bibitem{Shelah:74}
S.~Shelah.
\newblock Infinite abelian groups, whitehead problem and some constructions.
\newblock {\em Israel Journal of Mathematics}, 18(3):243--256, 1974.

\bibitem{Shelah:75}
S.~Shelah.
\newblock A compactness theorem for singular cardinals, free algebras,
  whitehead problem and tranversals.
\newblock {\em Israel Journal of Mathematics}, 21(4):319--349, 1975.

\bibitem{Shelah:77}
S.~Shelah.
\newblock Whitehead groups may be not free, even assuming ch, i.
\newblock {\em Israel Journal of Mathematics}, 28(3):193--204, 1977.

\bibitem{Stein:51}
K.~Stein.
\newblock Analytische {F}unktionen mehrerer komplexer {V}er{\"a}nderlichen zu
  vorgegebenen {P}eriodizit{\"a}tsmoduln und das zweite {C}ousinsche {P}roblem.
\newblock {\em Mathematische Annalen}, 123(1):201--222, 1951.

\bibitem{Tsankov:12}
T.~Tsankov.
\newblock Unitary representations of oligomorphic groups.
\newblock {\em Geometric and Functional Analysis}, 22(2):528--555, 2012.

\bibitem{Wu:71}
Ta-Sun Wu.
\newblock On (ca) topological groups, ii.
\newblock {\em Duke Math. Journal}, 38:533--539, 1971.

\bibitem{Zielinski:21}
J.~Zielinski.
\newblock Locally {R}oelcke precompact {P}olish groups.
\newblock {\em Groups, Geometry, and Dynamics}, 15(4):1175--1196, 2021.

\end{thebibliography}

\def\cprime{$'$} \def\cprime{$'$}

\end{document}